\documentclass[11pt]{amsart}
\usepackage{amsmath, amssymb}
\usepackage{amscd}
\usepackage{verbatim}
\usepackage[latin1]{inputenc}
\usepackage{url}
\usepackage{graphicx}
\usepackage{color}

\theoremstyle{plain}
\newtheorem{theorem}{Theorem}[section]
\newtheorem{proposition}[theorem]{Proposition}
\newtheorem{corollary}[theorem]{Corollary}
\newtheorem{lemma}[theorem]{Lemma}

\theoremstyle{definition}
\newtheorem{definition}[theorem]{Definition}
\newtheorem{example}[theorem]{Example}
\newtheorem{remark}[theorem]{Remark}

\theoremstyle{remark}

\numberwithin{equation}{section}

\newcommand{\N}{\mathbb N}
\newcommand{\Z}{\mathbb Z}
\newcommand{\Q}{\mathbb Q}
\newcommand{\R}{\mathbb R}
\newcommand{\C}{\mathbb C}

\newcommand{\A}{\mathbb{A}}
\newcommand{\p}{\mathfrak{p}}

\newcommand{\ai}{\mathfrak{a}}
\newcommand{\Pl}{\mathrm{Pl}}

\newcommand{\FS}{\mathrm{FS}}

\newcommand{\GL}{\mathrm{GL}}
\newcommand{\PGL}{\mathrm{PGL}}
\newcommand{\SL}{\mathrm{SL}}

\newcommand{\SO}{\mathrm{SO}}

\newcommand{\Vol}{\operatorname{Vol}}
\newcommand{\sign}{\operatorname{sign}}
\newcommand{\ba}{\backslash}

\title [Eigenvalues of Hecke and Casimir operators for Hilbert Maass forms] %for $\tau$-spectra]
{Joint distribution of eigenvalues of Hecke and Casimir operators  for Hilbert Maass forms}

\author{Roberto J. Miatello and Angel Villanueva}
\address{CIEM--FaMAF (CONICET), Universidad Nacional of C\'ordoba, Medina Allende s/n, Ciudad Universitaria, 5000 C\'ordoba, Argentina.}
\email{miatello@famaf.unc.edu.ar}
\email{villanueva@famaf.unc.edu.ar}
\thanks{This research was supported by grants from CONICET and FONCyT}
\date{\today}

\begin{document}

\begin{abstract}
	Let $F$ be a totally real number field, $\mathcal{O}_{F}$ the ring of integers,
	$\mathfrak a$ and $\mathfrak I$  integral ideals and let $\chi$ a character of $\A_F^\times/F^\times$.
	For each prime ideal  $\mathfrak{p}$ in $\mathcal{O}_{F}$, $\mathfrak{p}\nmid \mathfrak{I}$ let $T_{\mathfrak{p}}$ be the Hecke operator acting on  the space of Maass cusp forms  on
	 $L^2(\GL_{2}(F) \ba \GL_{2}(\A_F))$.
	
	In this  paper we  investigate the   distribution of joint eigenvalues of the Hecke operators $T_{\mathfrak{p}}$ 
	and  of the  Casimir operators $C_{j}$ in each archimedean component of $F$, for $1\le j \le d$. 
	Summarily, we prove that given a family of expanding compact subsets $\Omega_{t}$ of $\R^{d}$ as $t \rightarrow \infty$, and an interval $I_{\mathfrak{p}} \subseteq [-2,2]$, 
	then, if  $\mathfrak{p} \nmid \mathfrak{I}$ is a square in the narrow class group of $F$,  there are infinitely many  automorphic forms having eigenvalues of $T_{\mathfrak{p}}$ in $I_{\mathfrak{p}}$,  distributed on $I_{\mathfrak{p}}$ according to  a polynomial multiple of the Sato-Tate measure  and  having  their Casimir eigenvalues in the region $\Omega_{t}$, distributed according to the Plancherel measure. 
	
	Our results extend results of Serre \cite{Se97}, Knightly--Li (\cite{KL08}, \cite{KL13}) and  Bruggeman--Miatello \cite{BM13}.

	%\vfill

\end{abstract} 

\maketitle

\tableofcontents

\section{Introduction}

Let $k$ be even and let $f \in S_{k}(\Gamma_{0}(N))$ be a  holomorphic  cuspidal modular form that is a Hecke  eigenfunction. For each  prime $p\nmid N$ let  $\lambda_{p}(f)$ be the   normalized Hecke eigenvalue  defined by
\begin{equation}\notag
p^{k-1/2}T_{p}f=\lambda_{p}(f)f.
\end{equation}

The  Ramanujan--Petersson conjecture,  proved by Deligne \cite{De73},  implies that $|\lambda_{p}(f)|\leq 2$ for any $f \in S_{k}(\Gamma_{0}(N))$. 
%Inspired by the  Sato--\!Tate  conjecture, 
Serre studied the asymptotic  distribution of the Hecke eigenvalues $\lambda_{p}(f)$, when $f$ is fixed and   $p$ varies  and conjectured that for any   $f \in S_{k}(\Gamma_{0}(N))$, as $x \rightarrow \infty$, 
the $\lambda_{p}(f)$ for $p \leq x$ are equidistributed in $[-2,2]$  with respect to the  Sato--Tate measure 
\begin{equation}\label{eq:SatoTate}
d\mu _{\infty} (x)=\begin{array}{ll}
 \Bigg\lbrace \! 
    \begin{array}{ll}
      \frac{1}{\pi}\sqrt{1- \frac{x^{2}}{4}}dx \; \textrm{ if } \; x\in [-2,2],  \\ 0  \; \textrm{ other $x$}.
    \end{array}
\end{array}
\end{equation}
%also called the   Sato--\!Tate conjecture or the horizontal  problem  of Sato--\!Tate. 
This conjecture was proved by Barnet--Lamb, Geraghty, Harris and Taylor in \cite{BGHT11}.

Another point of view is the so called vertical problem  of Sato--\!Tate. Namely, fixed a prime $p$, when the level  $N$ and/or the weight $k$ vary, the eigenvalues  $\lambda_{p}(f)$ for $f \in  S_{k}(\Gamma_0(N))$, follow distribution laws with respect to the Sato--\!Tate measure. In this direction Serre proved the following

\begin{theorem} \cite{Se97} Let $N,k$ be  positive integers such that $k$ is even, $N+k \rightarrow \infty$ and $p$ is prime,  $p \nmid N$. Then the  
normalized Hecke eigenvalues $\lambda_p(f_{k,N})$ are equidistributed in the interval $\Omega=[-2,2]$  with respect to the  $p$-adic  Sato--\!Tate measure:
\begin{equation}\notag
\mu_{p}:= \frac{p+1}{\pi} . \frac{(1-x^{2}/4)^{1/2}}{(p^{1/2}+p^{-1/2})^{2}-x^{2}}dx.
\end{equation} 
\end{theorem}

%In his paper Serre gives several interesting applications  of this result, in particular, to  graphs and elliptic  curves. 

The distribution of the Hecke eigenvalues has been investigated  by several authors,  generalizing or  supplementing   Serre's results. 
For instance, in \cite{KL08},  Knightly--Li give a  result  on weighted equidistribution for holomorphic  forms over $F$ a totally real number field,  using a  polynomial multiple of the Sato--\!Tate measure and
%\footnote{Hay que decir bien lo de la satotate y sus multiplos.}
further,  in \cite{KL13}, they extended the result %of (weighted) equidistribution 
to the case of  Maass forms over $\Q$.
%, using a polynomial multiple of the Sato--\!Tate measure.
%, with equidistribution weighted by Fourier coefficients. 
For totally real number fields $F$, %Bruggeman--Miatello 
in \cite{BM13} a result of joint distribution of eigenvalues of the  Hecke operators $T_{\p^{2}}$ and the Casimir operators $C_{j}$ in each archimedean component of $F$, is proved  using a modification of the Sato--\!Tate measure. A recent far reaching generalization is due to Finis-Matz (\cite{FM19}) who use the trace formula to prove distribution results to any reductive group split over $\Q$ (see also \cite{MT15}  and \cite{KST20}).

In this paper we extend the distribution results in   \cite{BM13} for $F$ totally real, to  %the  eigenvalues$\lambda_{\varpi,\mathfrak{p}}$ of 
Hecke operators $T_{\p}$ for $\p$ 
a prime 
%such that $\p$ 
that is a square in the  narrow class group, using a polynomial multiple of the  Sato--\!Tate measure. %that coincides with the Sato-Tate measure a.e. in $\p$.

More precisely, given a family of compact boxes $\Omega_{t} =\prod_{j=1}^d \Omega_{t,j} $ of $\R^{d}$  satisfying some mild conditions (see \eqref{eq: Omega general}), where  $\Omega_{t}$ expands  in at least one component as $t \rightarrow \infty$,  and an interval $I_{\mathfrak{p}} \subseteq [-2,2]$ for $\mathfrak{p} \nmid \mathfrak{I}$ a prime ideal that is a square in the narrow class group, for any $r \in \ai^{-1}\mathfrak d^{-1}$, $\mathfrak d$ the inverse different of $F$, we prove the estimate
$$\sum_{f \in \mathcal B_{\chi,q}: \lambda(f)\in \Omega_t \atop \lambda_{\mathfrak{p}}(f) \in I_{\mathfrak{p}}}  |c^{\ai,r}(f)|^{2} =  \frac{2^d\sqrt{D_{F}}}{\pi^{d}h_F} \Phi_{\ai,r}(I_{\mathfrak{p}})\textrm{Pl}(\Omega_{t}) + \textrm o(V_{1}(\Omega_{t}))$$ %\Vol(\Gamma \ba \GL_{2}(\R)^{d})
where 
$\mathcal B_{\chi,q}$ is  an orthonormal basis of the space 
$L^{2,\textrm {disc}}(\GL_2(F) \ba \GL_2(\A_F)/K_0(\mathfrak I))_{\chi, q}$ (satisfying the conditions in  Definition~\ref{def: BFSchiq}), 
$c^{\ai,r}(f)$ is the  $(\ai,r)$-Fourier  coefficient of  $f$, $\lambda(f)=(\lambda_j(f))$ where $\lambda_j(f)$ is the eigenvalue of $C_j$ and, finally, $\lambda_\p (f)$ is the eigenvalue of $T_\p$ on $f$. Furthermore,  $\Phi_{\ai,r}(x)$ is a  polynomial multiple of the Sato-Tate measure (see \eqref{eq: phi}) that coincides with the Sato-Tate measure for any $r$  such that  ${\p}\nmid r\ai \mathfrak{d}$,
Pl denotes the  Plancherel measure and $V_{1}$ is a measure on $\R^d$ such that $V_{1}=\textrm O(\textrm{Pl})$.

This  implies that
$$\lim_{t\rightarrow \infty} (\textrm{Pl}(\Omega_{t}))^{-1} \sum_{f \in \mathcal B_{\chi,q}: \lambda(f)\in \Omega_t \atop \lambda_{\mathfrak{p}}(f) \in I_{\mathfrak{p}}}  |c^{\ai,r}(f)|^{2} =   \frac{2^d\sqrt{D_{F}}}{\pi^{d}h_F} \Phi_{\ai,r}(I_{\mathfrak{p}}). $$% \Vol(\Gamma_{0}(\mathfrak{I}) \ba \GL_{2}(\R)^{d})
%\textrm{Pl}(\Omega_{t}) + o(V_{1}(\Omega_{t}))\right]$$cuando la medida of $\Omega_{t}$ tiende a infinito, 
In particular, this says that there are infinitely many  automorphic eigenforms with $T_{\mathfrak{p}}$ eigenvalues    distributed according to  
%a polynomial multiple of the Sato-Tate measure 
the measure $\Phi_{\ai,r}$ of   $I_{\mathfrak{p}}$, and   with  Casimir eigenvalues in the given region $\Omega_{t}$, distributed according to the Plancherel measure of  the region. 

 To prove the main results (Theorem~\ref{thm:principal 1} and Theorem~\ref{thm: principal}), we use  the Kuznetsov sum formula and, as a main tool, an asymptotic formula similar to one proved in \cite{BM10}. In the proofs, we use results in \cite{Ve04}, \cite{BM10} and \cite{BMP03}.

As an application, we obtain   results on the distribution of eigenvalues of holomorphic Hilbert modular forms (Thm \ref {thm: caso holomorfo}) and  on  weighted equidistribution of  Hecke eigenvalues (Thm \ref{thm: equidistribucion con peso}).

\begin{remark} Although we restrict this paper to $F$ totally real, part of the argument works for any field $F$, using the Kuznetsov formula in \cite{Ma13}. It seems still non trivial to generalize the estimate $o(V_1(\Omega_t))$ of the remainder in \eqref{eq:asymptotic},  proved in \cite{BM10} in the totally real case.
\end{remark}
\subsection*{Acknowledgments}
The authors wish to thank R. Bruggeman, G. Harcos and A. Knightly for  useful comments on a first version of this paper.

%\footnote{Cambié todos los $\lambda_j(f)\in \Omega_t$ por $\lambda(f)\in \Omega_t$. }

\section{Preliminaries} \label{sec:preliminaries}

In this section we introduce some basic notions and notations that  will be needed throughout this paper.  

Let  $F$ be a  totally real number field, $[F : \Q]=d$ and let $\mathcal{O}_{F}$ be the ring of integers of $F$.
Let $\sigma^{(j)}: F \rightarrow \R$ ($j = 1, \ldots, d$) be the field embeddings from  $F$ into $\R$ and denote by
$ \sigma: \GL_2(F) \rightarrow  \GL_2(F_\infty) \cong \GL_2(\R)^d$  the canonical embedding 
given by
\begin{equation}\label{eq:Fembedding} %i
\sigma(\left[\begin{smallmatrix}a&b\\c&d\end{smallmatrix}\right]) = %\left\{
(\sigma^{(1)}\left[\begin{smallmatrix}a&b\\c&d\end{smallmatrix}\right],\ldots,\sigma^{(d)}\left[\begin{smallmatrix}a&b\\c&d\end{smallmatrix}\right]).\end{equation}

 %Given a  ring  $R$, 
Denote, for $x, \theta \in \mathbb R^d$, $t,  y \in (\mathbb R^{\times})^d$, if $x=(x_j)$, $t=(t_j)$, $\theta = (\theta_j)$, $y=(y_j)$:
%$y \in (\R^{\times})^{d}$,  and $\theta_j \in \mathbb{R}$, $1\le j \le d$,
\begin{align}
n(x) &= \left( \left( \begin{matrix}
1 & x_1\\
0 & 1
\end{matrix}\right), \ldots , \left( \begin{matrix}
1 & x_d\\
0 & 1
\end{matrix} \right)\right)\\
z(t) & =\left( \left( \begin{matrix}
t_1 & 0\\ 
0 & t_1
\end{matrix} \right),\ldots,\left( \begin{matrix}
t_d & 0\\ 
0 & t_d
\end{matrix} \right)\right) \\
a(y)&=\left( \left( \begin{matrix}
y_{1} & 0\\
0 & 1
\end{matrix} \right),\ldots, \left( \begin{matrix}
y_{d} & 0\\
0 & 1
 \end{matrix} \right) \right), 
\\
k(\theta) &= \left( \left( \begin{matrix}
\cos (\theta_{1}) & \sin (\theta_{1})\\
-\sin (\theta_{1}) & \cos (\theta_{1})
\end{matrix} \right),\ldots, \left( \begin{matrix}
\cos (\theta_{d}) & \sin (\theta_{d})\\
-\sin (\theta_{d}) & \cos (\theta_{d})
\end{matrix} \right) \right) .
\end{align}

We will make use of the following subgroups of $\GL_2(\R)^{d}$.
Let $N= \lbrace n(x) : x\in \mathbb{R}^{d} \rbrace$, $A= \lbrace a(y) : y \in (\R ^{\times})^{d} \rbrace$, $K_\infty = \lbrace k(\theta) : \theta \in \mathbb{R}^{d} \rbrace   \cong \textrm{SO}(2)^d$,
 $Z_\infty = \lbrace z(t) : t \in (\R^{\times})^d\rbrace$, the center  of $\GL_2(F_\infty) \cong \GL_2(\R)^d$. 
 %((\zeta_1),\ldots,h(\zeta_d))$ with $\zeta_j \in {\R^\times}$ for $1\le j \le d$. 
Then, one can write every $g \in \GL_2 (\R) ^d$   uniquely as 
\begin{equation}
\label{eq:Iwasawa} g = znak, 
\end{equation}
where $z \in Z_\infty, n \in N, a \in A$ and $k \in K_\infty$, hence $\GL(2,\R)^d \simeq NAK_\infty Z_\infty$.
%By the Iwasawa decomposition we have  that $ \GL_2(\R)^d \simeq NAK Z_\infty$, and the map 
%$(n,a,k,z) \rightarrow nakz $ is a diffeomorphism from $NAK Z_\infty$ onto  $\GL_2(\R)^d$.
%If $R$ is a subring of $\mathbb{R}^{d}$ then we set $N(R)=\lbrace n(x) : x\in R \rbrace$.

\

The  function $S: \R^d \rightarrow \R$ given by $S(x) = \sum_{j} x_{j}$ extends the trace function $\textrm{Tr}_{F/\Q} : F \rightarrow \Q$. 
%since $\sum_1^d x^{\sigma_j} = \textrm{Tr}_{F/\Q}  (x)$ for $x \in F$. 
The   inverse different ideal of $F$ is the fractional ideal  \begin{equation}
\mathfrak{d}^{-1}=\lbrace x \in F \; : \; S(x\xi) \in \Z \; \textrm{for every} \; \xi \in \mathcal{O}_{F} \rbrace.
\end{equation}

%The Hilbert modular group is the subgroup $ \sigma(\GL_2(\mathcal O_F))$ of $\GL_2(\R)^d$. 
Given  integral ideals $\mathfrak I$ and $\ai$ of $\mathcal O_F$, %and an integral ideal $\ai$ %we set 
%we define the Hecke 
consider the congruence subgroup of $\GL_2(F_\infty)$:
%of the Hilbert modular group $\sigma(\GL_2(\mathcal O_F))$ 
%given by 
\begin{align}
\Gamma_0(\mathfrak I,\ai) &= \left\{  \sigma(\left[\begin{smallmatrix}a&b\\c&d\end{smallmatrix}\right])
%(\sigma^{(1)}(\left[\begin{smallmatrix}a&b\\c&d\end{smallmatrix}\right]),\ldots,\sigma^{(d)}(\left[\begin{smallmatrix}a&b\\c&d\end{smallmatrix}\right]))
: \left[\begin{smallmatrix}a&b\\c&d\end{smallmatrix}\right] \in \GL_2(F), ad-bc \in \mathcal{O}_F^{\times}, \,\,\,
a,d \in \mathcal O_F,\,\,\, b\in \ai,\,\, c \in \ai^{-1} \mathfrak I \right\}.
\end{align}
When  $\ai = \mathcal O_F$, then $\Gamma_0(\mathfrak I,\ai) = \Gamma_0(\mathfrak I).$
 We will  also make  use of  the subgroup
$\Gamma(\mathfrak{I}, \mathfrak{a})_{N} = N(F) \cap \Gamma_0 (\mathfrak{I}, \mathfrak{a})$ of $N(F)$.

Let  $\mathbb{A}_{F}$, $\mathbb{A}_{F,f}$ and $\mathbb{A}^{\times}_{F}$ be the ring of adeles, the ring  of finite adeles and the group of ideles of $F$. 
%Let $\mathbb{A}^{\times}_{F}$ be the group of  ideles of $F$.
%units of  $\mathbb{A}_{F}$ 
%and let $\mathbb{A}_{F,f}$ be the ring of finite adeles. 
We have canonical embeddings   of $F$ 
into $\mathbb{A}_{F}$:
\begin{equation*}
i: F \rightarrow \mathbb{A}_{F}, \quad i_{\infty}: F \rightarrow \mathbb{A}_{F_\infty}, \quad i_{f}: F \rightarrow \mathbb{A}_{F,f}\,.
\end{equation*}
Denote by $\mathcal{C}_F$ and  $\mathcal{C}_{F}^{+}$ the class group and  narrow class group of $F$,  and  by $h_{F}$ and $h_{F}^+$  the class number and narrow class number of $F$ respectively. 
%that is, $\mathcal{C}_{F}=\tfrac{\mathcal{I}_{F}}{\mathcal{P}_{F}}$  the quotient group of fractional  ideals $\mathcal{I}_{F}$ of  $F$  by the subgroup of   principal  ideals $\mathcal{P}_{F}$.  Let $\mathcal{C}_{F}^{+}=\tfrac{\mathcal{I}_{F}}{\mathcal{P}_{F}^{+}}$, the narrow class group of $F$,  where $\mathcal{P}_{F}^{+}$ is the group of principal ideals generated by totally positive elements.
%Then $\mathcal{C}_{F}$ and $\mathcal{C}_{F}^+$ are  finite abelian groups with $h_{F}$ and $h_{F}^+$ elements respectively.

To each fractional ideal $\mathfrak{a}$ of $F$ we  associate an idele $\pi_{\mathfrak{a}} \in \mathbb{A}_{F}^{\times}$ which at each  non-archimedean place has the same valuation as $\mathfrak{a}$. Such an idele  is unique up to multiplication by units. Let $\mathfrak a_1,\ldots, \mathfrak a_{h_F}$ be representatives of the ideal classes in $\mathcal C_F$ and let $\pi_{\mathfrak a_1},\ldots, \pi_{\mathfrak a_{h_F}}$ be corresponding ideles.

%We  fix   representatives $\mathfrak{a}_{1}, \ldots, \mathfrak{a}_{h}$ of the 
%class group $\mathcal{C}_{F}$ %^{+}$
% with $\ai_{j}$  an integral ideal for each  $j \in \lbrace 1, \ldots, h \rbrace $, and   corresponding  ideles  $\pi_{\mathfrak{a}_{1}}, \ldots, \pi_{\mathfrak{a}_{h}}$. As a representative of the  identity element we choose  $\mathfrak{a}_{1} = \mathcal{O}_{F}$ and we let $\pi_{\mathcal{O}_{F}} = i(1)$. 
%If $\mathfrak{p}$ is a prime ideal in $\mathcal{O}_{F}$ such that $\p \nmid \mathfrak{I}$ and $F_{\mathfrak{p}}$, $\mathcal{O}_{\mathfrak{p}}$, are the  completion of $F$ and $\mathcal{O}$ at $\mathfrak{p}$ respectively, let $\pi_{\mathfrak{p}}$ be the uniformizer of the local ring $\mathcal{O}_{\mathfrak{p}}$.  {\red Then we may take $i_f(\pi_\p)$ as the idele associated to the ideal $\p$.}\footnote{Revisar esto muy bien.}

%\subsection{Congruence subgroups}

%We now  define certain  open compact subgroups of $\GL_2(\A_{F, f}) $.
Let $v$ be a place of $F$,  let $F_v$ be  the completion of $F$ at $v$ and $\mathcal{O}_v$ the integral subring of $F_v$.  
We set $F_\infty := \prod_{v|\infty} F_v \cong \R^d$. 
  
Given  $\mathfrak{I}$ an  integral  ideal and $\ai$ a  fractional ideal, for each finite place  $v$ of $F$ we
consider the  open compact subgroup of $\GL_2(\A_{F, f}) $
\begin{equation}
K_{0,v}(\mathfrak{I}_{v}, \mathfrak{a}_{v}) = \left\{ \left( \begin{matrix}
a & b\\
c & d
\end{matrix}\right) \in \GL_2(F_v) \; : \; ad-bc \in \mathcal{O}_{v}^{\times} , \; a,d \in \mathcal{O}_v, c\in \mathfrak{a}_{v}^{-1} \mathfrak{I}_{v}  , \; b \in \mathfrak{a}_{v} \right\}.
\end{equation}

Let $K_{0}(\mathfrak{I}, \ai) := \prod_{v} K_{0,v}(\mathfrak{I}_{v}, \mathfrak{a}_{v})$. In case  $\ai = \mathcal O_F$, we write $K_{0}(\mathfrak{I})$ in place of  $K_{0}(\mathfrak{I}, \mathcal O_F)$. 
% In general, the use of $\ai_{v}$ is necessary to deal with fields $F$ of narrow class number greater than 1.??

There is the following relation between  $\Gamma_{0}(\mathfrak{I}, \mathfrak{a})$ and $K_0(\mathfrak{I}, \ai)$
\begin{equation} 
i_f(\Gamma_{0}(\mathfrak{I}, \mathfrak{a})) =   K_0(\mathfrak{I}, \ai) \cap i_f(\GL _2 (F)).
\end{equation}
%where the intersection is taken in   $\GL_2(\mathbb{A}_F)_f$.
%, that is, $\Gamma_{0}(\mathfrak{I}, \ai)$ is the subgroup of $\GL_2(F)$ such that $i_{f}(\Gamma_{0}(\mathfrak{I}, \ai)) = K_0(\mathfrak{I}, \mathfrak{a})\cap i_f(\GL_2(F))$.

We now recall strong approximation for number fields with %narrow 
class number   $h_F$.  %(see \cite{Ge88} for instance). % (narrow class number). 
 
\begin{theorem} 
Let $\mathfrak{a}_{1}, \ldots, \mathfrak{a}_{h_F}$ be integral representatives  of the 
class group $\mathcal{C}_{F}$, with 
% with $\ai_{j}$  an integral ideal for each  $j \in \lbrace 1, \ldots, h \rbrace $, 
corresponding  ideles  $\pi_{\mathfrak{a}_{1}}, \ldots, \pi_{\mathfrak{a}_{h_F}}$. Then one has the following decompositions 
\begin{align}
&\mathbb{A}_F = i(F) + (F_{\infty}  \prod_{\p < \infty}\mathcal{O}_\p ), \quad\quad
%\\&
\mathbb{A}^{\times}_F= \bigsqcup _{i=1}^{h_F}\,  i(F^{\times}) \pi_{\mathfrak{a}_{i}} \big(F^{\times}_{\infty}  \prod_{\p<\infty} \mathcal{O}^{\times}_{\p}\big),
\\& \GL_2 (\mathbb{A}_F) = \bigsqcup _{i=1}^{h_F} i(\GL _2 (F))\left(\begin{matrix}
\pi_{\mathfrak{a}_{i}} & 0 \\
0 & 1
\end{matrix}\right)  \GL_2 (F_\infty)K_f. \label{strongapprox}
\end{align}
 where $K_f$ is any compact open subgroup of $\GL_2(A_F)_f$ such that $\det$ is onto  $\mathcal O_F^\times$. From now on we use $K_f$= $K_0(\mathfrak I)$.
%\footnote{Hacen falta los parÃ©ntesis a la der de $K_f$ y a la izq de $GL$?.}
%where $K_f$ is any compact open subgroup of $\GL_2 (\A_{F, f})$. 
%, in all cases, $F$, $F^{\times}$ and $\GL_2(F)$ are embedded diagonally in $\mathbb{A}_F$, $\mathbb{A}_{F}^{\times}$ and   $\GL_2(\mathbb{A}_F)$ respectively.
\end{theorem}
If $F$ has class number 1, one needs only consider the case $\mathfrak a= \mathcal{O}_F$. %$\ai_{v}= \mathcal{O}_F$.

%Let $F$ be a number field, and let $\mathbb{A}_{F,f}$ be the ring of finite adeles  of $F$. 
Furthermore, by strong approximation for the group $\SL_2$ one has 
\begin{equation}\label{strongapproxSL2}
\SL_2(\mathbb{A}_{F,f}) = i_f(\SL_2(F))K_f
\end{equation}
for any compact open  subgroup  $K_f$  of $\SL_2(\mathbb{A}_{F,f})$. We will make use of this fact quite often in Sections 4 and 5. 

%\footnote{En la 2.0.1 no se explica bien la relaci\'on entre el car\'acter ad\'elico $\chi$, el $\chi_\infty$ y el class character que m\'as tarde se usar\'a en la Kuznetsov. Escribir algo nuevo acaÂ´. Angel, creo que hay que agregar que el caracter satisface $\psi=1$ en Zinfty y entonces las funciones cl\'asicas estan en $(PGL_2)^d$, como vos dijiste.}

%\end{proof} 

%\begin{remark} Notar la diferencia  with (\ref{TAF}).
%\end{remark}
\subsubsection{Characters.}\label{ss:characters}
We fix $\chi$ a  Hecke character  %Gr{\"o}ssencharacter 
of $\A_F^{\times}/F^{\times}$, $\mathfrak I$ an ideal  divisible by the conductor of $\chi$ and we denote by $\chi_v, \chi_f, \chi_{\infty}$ the restrictions of $\chi$ to $F^{\times}_v$, $\A_{F,f}^{\times}$ and $F_{\infty}^{\times}$ respectively.
Then $\chi$ induces a character of $K_0(\mathfrak{I}, \ai)$,  still denoted by $\chi_f$, given by
\begin{equation}\label{eq:chi_f}
\chi_f\left(\left(\begin{smallmatrix}
a & b\\
c & d
\end{smallmatrix}\right)\right) =  \prod_{v} \chi_v (d_{v})
\end{equation}
and a character of  $\Gamma_0(\mathfrak{I}, \ai)$, defined by $\chi(\gamma)= \chi( i_f(\gamma)$).  We  shall assume that $\chi_{\infty}=1$ and 
\begin{equation}\label{eq:qtipo}                                                         
\phi_{q}(k(\theta))= e^{iS(q\theta)}, \; \textrm{for} \; k(\theta) \in \SO_{2}(\R)^{d}\simeq K_\infty.
\end{equation}

\subsubsection{Measures} 
As in \cite{BH10}, on  $\A_F$ we use the measure $dx$ that is the product of
the normalized Lebesgue measure
$\pi^{-d}dx_1 \ldots dx_d$ on $F_\infty$ and the Haar measure on $F_\p$ so that $\mathcal{O}_\p$ has measure $1$, for each $\p$. It induces the Haar measure on $ \A_F/F$ with total measure one.
Also, 
on $\A_F^{\times}$ we use  the measure $d^{\times}y$ that is the product of the  Haar measure $(dy_1/|y_1|) \ldots (dy_d/|y_d|)$ 
on $F_\infty^{\times}\cong {\R^\times}^d$ and  the Haar measure on $F_\p^\times$ so that $\mathcal{O}_\p^{\times}$ has measure $1$. 
%On $\A^{\times}$ we use  the product of these measures $d^{\times}y$, 
It induces a Haar measure on $ \A_F^{\times}/ F^{\times}$. On $K = \SO_2(\R)^d K_0(\mathfrak{I})$ and its factors we
use  Haar probability measures and on $Z(F_\infty)\ba \GL_2(F_\infty)$  the Haar measure which satisfies 
\begin{align*}
\int_{Z(F_\infty)\ba \GL_2(F_\infty)} f(g)dg = \int_{F_\infty^{\times}} \int_{F_\infty}\int_{\SO_2(F_\infty)} f\left( \left(\begin{matrix}
y & x \\
0 & 1
\end{matrix} \right) k \right) dk dx \frac{d^{\times}y}{|y|}.
\end{align*}

On $\GL_2(F_\p)$ we fix the Haar measure so that $K_{0,\p}(\mathcal{O}_\p)$ has measure $1$ and on $Z(F_\infty)\ba \GL_2(\A_F)$ we use the product measure. This induces the Haar measure on $Z(\A_F^\times)\ba \GL_2(\A_F)$ satisfying 
\begin{align*}
\int_{Z(\A_F^\times)\ba \GL_2(\A_F)} f(g)dg = \int_{\A_F^{\times}} \int_{\A_F}\int_{\SO_2(F_\infty)K_0(\mathfrak{I})} f\left( \left(\begin{matrix}
y & x \\
0 & 1
\end{matrix} \right) k \right) dk\, dx \, \frac{d^{\times}y}{|y|}.
\end{align*}

\subsection{Adelic and classical square integrable automorphic forms}
As in \cite{Ve04} and in \cite{BH10}, we consider the subspace  FS  
%on $\GL_2 (\mathbb{A}_{F})$ transforming  by $K_{0}(\mathfrak{I})$ on the right by the character $\chi_f$,   that are  square integrable  on $\GL_2 (F)\backslash \GL_2 (\mathbb{A}_{F})/Z( \mathbb{A}^\times_{F})$.
 %That is, %the space of
of functions $f: \GL_2 (\mathbb{A}_{F}) \rightarrow \C$ such that
\begin{itemize}
		\item[(i)] $f\left(\gamma g \left( \begin{smallmatrix}
	z & 0\\
	0 & z
	\end{smallmatrix}\right)k_{0}\right) = f(g)  \chi_f(k_{0})$ for every  $g \in \GL_2(\A)$, $\gamma \in \GL_2 (F)$, $z \in F_{\infty}^{\times}$ and
	 $k_{0}\in K_{0}(\mathfrak{I})$,
\item[(ii)] $\int_{\GL_2 (F) \ba \GL_2 (\mathbb{A}) / Z(\mathbb{A}_{F}^\times)} |f(g)|^{2}dg < \infty$.
\end{itemize}

We will also make use of the subspace $\textrm {FS}_\chi$ of functions in $\textrm {FS}$ satisfying the additional condition that $f\left(g\left( \begin{smallmatrix}
z & 0\\
0 & z
\end{smallmatrix}\right)\right)= f(g)\chi(z)$ for any $z\in \A_F^\times$.

There is a  standard correspondence between the space $\FS$ %$L^{2}\big(\GL_2 (F)\ba \GL_2 (\mathbb{A}_{F}),\chi \big)$ 
of adelic automorphic forms and spaces of classical automorphic forms. 
One has an isomorphism (see \cite[(93)]{BH10}) %for every $q\in \Z^d$ 
\begin{equation}\label{eq:L2classical}
%L^{2} (\GL_{2} (F) \ba \GL_{2} (\mathbb{A}_{F}),\chi)  
\textrm {FS} \cong \oplus_{i=1}^{h_F} L^{2}(\Gamma_{0}(\mathfrak{I},\ai_{i}) \ba \PGL_{2} (\R)^d, \chi^{-1}).
\end{equation}

Here, for any  fractional ideal $\ai$  in $F$,  $L^{2}\big(\Gamma_{0}(\mathfrak{I}, \ai) \ba \PGL_2 (\R)^d , \chi^{-1}\big)$ is the  completion of the  space of smooth functions $f: \GL_2(\R)^d \rightarrow \C$ such that
\begin{itemize}
	\item[(i)]   $f(\gamma g \left( \begin{smallmatrix}
	z & 0\\
	0 & z
	\end{smallmatrix}\right))= \chi(\gamma)^{-1} f(g)$, for every   $\gamma \in \Gamma_0 (\mathfrak{I}, \ai)$ and   $z \in F_{\infty}^{\times}$. 
	Here, $\chi(\gamma) := \chi_f (i_f (\gamma))$, % k \in K_{\infty}$, \footnote{Ver.}
	\item[(ii)] $\int_{\Gamma_{0}(\mathfrak{I},\ai) \ba \GL_2 (\R)^d / Z_{\infty}^\times} |f(g)|^{2}dg < \infty$.
\end{itemize}

%In this  section we will give a correspondence between classical and adelic Hilbert automorphic forms.

%\footnote{Referencia?. Abajo agregue la accion de $K_{\infty}$ por $\phi_{q}$.}
%\footnote{En (2.21, 2.22) se divide por Z(F infty) a la derecha pero no a la izquierda, ??}

%\footnote{Revisar esta correspondencia, considerando tb el class group character (si corresponde (?))}

We recall the correspondence in \eqref{eq:L2classical}. If $f\in$ FS,
 %$f \in L^{2} (\GL_{2} (F){\blue Z_\infty} \ba \GL_{2} (\mathbb{A}_{F}), \chi,\omega)$, 
 let  $f_{\mathfrak{a}_{j}} $ in $L^{2}\big(\Gamma_{0}(\mathfrak{I}, \ai_{j}) \ba \PGL_2 ( \R)^d, \chi^{-1} \big)$ be  given by    
\begin{equation} \label{classical components}
f_{\mathfrak{a}_{i}} (g_{\infty}) = \overline{\chi (\pi_{\mathfrak{a}_{i}})} f \left( \left( \begin{matrix}
\pi_{\mathfrak{a}_{i}} & 0\\
0 & 1
\end{matrix} \right)g_{\infty} \right), \textrm{ for } i=1, \ldots, h_F.
\end{equation}
Note that $f_{\mathfrak{a}_{i}}$ is independent of the choice of the idele $\pi_{\mathfrak{a}_{i}}$.

Conversely, given $f_{\mathfrak a_i}$ in $L^{2}\big(\Gamma_0 (\mathfrak{I}, \ai_{i}) \ba \PGL_2( \R)^d, \chi^{-1}
\big)$ for $i = 1, \ldots, h_F$, we can recover   $f \in \textrm{FS}$
%L^{2}(\GL_2(F) \ba \GL_2(\mathbb{A}_{F}),\chi,\omega)$ 
by defining its  value in each component by
\begin{align*}
f\left( \gamma \left(\begin{matrix}
\pi_{\ai_{i}} & 0\\
0 & 1
\end{matrix} \right)g_{\infty}  k_{0}\right):= \chi_f(k_{0}) f_{\mathfrak a_i}(g_{\infty})\chi(\pi_{\ai_i}),
\end{align*}
for $\gamma \in \GL_2 (F)$ and $g_{\infty} \in  \GL_2(\R)^d$.  %F_{\infty})$, $k_{0} \in K_{0}(\mathfrak{I})$ and $z \in \A^\times$. 
This map is well defined and both maps are the inverse of each other, thus we have \eqref{eq:L2classical}.
Furthermore, this correspondence is a $(\mathfrak g, K_\infty)$-morphism of the associated  $(\mathfrak g, K_\infty)$-modules. 

\
%\textbf{Automorphic representations.} 
%\footnote{Hay que entender bien el rol del centro, adÃ©lico y clÃ¡sico de GL2 y como influye en todas las cuentas. Ya puse unas footnotes sobre esto. HabrÃ­a que mirar Bump, Goldfeld-Hundley etc.}
%Let $L^2(\Gamma_0(\mathfrak{I}, \mathfrak{a}) \ba \GL_2(\R)^d, \chi)$ denote the Hilbert space of classes of functions transforming according to $\chi$   on the left, i.e. $f(\gamma g) = \chi(\gamma) f(g)$ for any $\gamma \in \Gamma_0(\mathfrak{I}, \mathfrak{a})$ and $g\in G$. 

The group $\GL_2(\R)^d$ acts unitarily on each Hilbert space  $L^2(\Gamma_0(\mathfrak{I}, \mathfrak{a}_i) \ba \PGL_2(\R)^d, \chi^{-1})$ by right translations and
%%%%%%%%%%% esto es para SL_2 nomÃ¡s %%%%%
% This space is split up according to central characters, indicated by $\xi \in \{0,1\}^d$. By $L^2_\xi (\Gamma_0(\mathfrak{I}, \mathfrak{a}) \ba G, \chi)$ we mean the subspace on which the center acts by 
%\begin{align*}
%\left(\left(\begin{matrix}
%\zeta_1 & 0 \\ 0 & \zeta_1
%\end{matrix} \right), \ldots, \left( \begin{matrix}
%\zeta_d & 0 \\ 0 & \zeta_d
%\end{matrix} \right) \right) \mapsto \prod_j \zeta_j^{\xi_j},
%\end{align*}
%where $\xi_j \in \{1,-1\}$. This subspace can be non-zero only if the following compatibility condition holds:
%\begin{align*}
%\chi(-1) = \prod_j (-1)^{\xi_j}.
%\end{align*}
%We assume this throughout this paper. \footnote{Esto es para SL2. Hay que ver Bump para GL2. Venkatesh usa el carÃ¡cter central trivial. Nosotros? Tal vez sÃ³lo hace falta con imagen en ${\pm 1}$}
there is an orthogonal decomposition
\begin{align}\notag
L^2 (\Gamma_0(\mathfrak{I}, \mathfrak{a}_i)\ba \PGL_2(\R)^d,\! \chi^{-1})\! = \! L^{2,\textrm{cont}} (\Gamma_0(\mathfrak{I}, \mathfrak{a}_i)\ba \PGL_2(\R)^d, \chi^{-1}) \!\oplus\! L^{2,\textrm{disc}} (\Gamma_0(\mathfrak{I}, \mathfrak{a}_i)\ba \PGL_2(\R)^d, \chi^{-1}).
\end{align}
The %$\GL_2(\R)^d$-
invariant subspace $L^{2,\textrm{cont}}(\Gamma_0(\mathfrak{I}, \mathfrak{a}_j)\ba \PGL_2(\R)^d, \chi^{-1})$ can be described by integrals of Eisenstein series and the orthogonal complement $L^{2,\textrm{disc}} (\Gamma_0(\mathfrak{I}, \mathfrak{a}_i)\ba \PGL_2(\R)^d, \chi^{-1})$ is an orthogonal direct sum of closed irreducible $\PGL_2(\R)^d$-invariant subspaces.
% If $\chi = 1$, the constant functions form an invariant subspace. 
%All irreducible invariant subspaces (except the constant functions when $\chi =1$) 
%have infinite dimension. They 
%are cuspidal and span the space $L^{2,\textrm{cusp}} (\Gamma_0(\mathfrak{I}, \mathfrak{a}_i)\ba \PGL_2(\R)^d, \chi^{-1})$.
%, the orthogonal complement of the constant functions in $L^{2,\textrm{disc}} (\Gamma_0(\mathfrak{I}, \mathfrak{a}_i)\ba \PGL_2(\R)^d, \chi^{-1})$. 

Furthermore, for each $\mathfrak{a}_i$ one has a decomposition 
\begin{equation} 
L^2 (\Gamma_0(\mathfrak{I}, \mathfrak{a}_i)\ba \PGL_2(\R)^d,\! \chi^{-1})
= \sum_\xi L^2_\xi (\Gamma_0(\mathfrak{I}, \mathfrak{a}_i)\ba \PGL_2(\R)^d,\! \chi^{-1})
\end{equation}
where $\xi$  runs through the characters of the group $M=\left\{\left(\left(\begin{smallmatrix}  \zeta_1 &0\\0&1 \end{smallmatrix}\right),\ldots, \left(\begin{smallmatrix} \zeta_d &0\\0& 1 \end{smallmatrix}\right)\right) \; : \; \zeta_j = \pm 1 \right\}$ satisfying the compatibility condition 
%$\xi(\eta_1,\ldots,\eta_d) =\prod_1^d (-1)^{\xi_j}$ and the $\xi_j\in \{0,1\}$  
$\chi(\left(\begin{smallmatrix}  -1&0\\0&1 \end{smallmatrix}\right),\ldots, \left(\begin{smallmatrix}  -1&0\\0&1 \end{smallmatrix}\right))=\prod_1^d (-1)^{\xi_j}=1$ with $\xi_j\in \{0,1\}$.

\subsection{Fourier terms}

We fix  a maximal orthogonal system $\{V_\varpi\}_\varpi$ of irreducible invariant subspaces of the Hilbert space $L^{2,\textrm{ disc}} (\Gamma_0(\mathfrak{I}, \mathfrak{a})\ba \PGL_2(\R)^d, \chi)$. %This system is unique if all $\varpi$ are inequivalent, but 
%in general, there will be
Each such subspace $\varpi$ has finite multiplicity and %, due to oldforms.
%furthermore , each irreducible automorphic representation $\varpi$ of $  \GL_2(\R)^d $  
splits as a tensor product $\varpi \cong \otimes_j \varpi_j$ of irreducible representations $\varpi_j$ of $\GL_2(\R)$, where $j$ runs over the $d$ archimedean places of $F$. 
%\footnote{Hemos puesto disc en vez de cusp, aunque, salvo de unas reps 1-dimensional, P.Maga y BHarcos dicen que todo es cuspidal.}

%\footnote{Later mention the decomposition in Blomer-Harcos (94)}

Set $\psi_{\infty}(x) = e^{2\pi i S(x)} = e^{2\pi i (x_1 + \ldots + x_d)}$.
Every    classical automorphic form $f$ for $\Gamma_0(\mathfrak I,\mathfrak a)$
has an expansion 
\begin{equation}
f(n(x_\infty)  g_\infty)= \sum_{r\in \ai^{-1} \mathfrak{d}^{-1} } {\psi}_{\infty}(rx_\infty)%\chi_{r}(n)
F_{\ai,r}f(g_\infty) \;\;\; (n \in N)
\end{equation}
where for  $r\in \mathfrak{a}^{-1} \mathfrak{d}^{-1} $ \begin{equation}
F_{\ai,r} f(g):= \frac{1}{\Vol (\R^d / \mathfrak a)} \int_{\R^d /\mathfrak a} e^{-2\pi i S(rx)} f(n(x)g)dx.
\end{equation}

A $(\Gamma, \chi)$-automorphic function is cuspidal if $F_{\ai,0} f = 0$. 
%\footnote{Revisar esto en BM09}
%\footnote{Para funciones $Z_\infty$-invariant only?.}
%\footnote{Parece que hace falta una f\'ormula para $F_{\ai,r}$ (?)}

If $\varpi = \otimes_{j} \varpi _{j}$ is irreducible with spectral parameter $\nu_\varpi$ and $f$ in $V_\varpi$ is of weight  $q \in \Z^d$, then 
%is a weight of $\varpi$, then  there exists  %an  orthogonal basis 
%, with each $({\blue f}_{\varpi, q})$ of weight $q$, such that  
\begin{equation}\label{fourierterm}
F_{\ai,r}( f)(g) = c^{\ai,r}( f)d^{\ai,r}(q, \nu_{\varpi}) W_{q}(r, \nu_{\varpi}; g)     % \;\;\; \textrm{ for any weight} \; q, \; \textrm{where}
\end{equation}
with
\begin{equation}\label{eq:d^r}
d^{\ai, r}(q,\nu):= \frac{1}{\sqrt{2^{d}|D_{F}N(\ai r)|}} \prod_{j=1}^{d} \frac{e^{\pi i q_{j}}}{\Gamma \left( \frac{1}{2}+\nu_{j}+ \frac{q_{j}}{2}\sign (r_{j}) \right)} \;\;\textrm{ and }
\end{equation}
\begin{equation}\label{whittaker}
W_{q}(r, \nu;\left(\begin{smallmatrix}z_\infty&0\\0& z_\infty\end{smallmatrix}\right) n(x_\infty)a(y_\infty)k_\infty) := {\psi}_{\infty}(rx_\infty)\phi_{q}(k_\infty) \epsilon_{\varpi}(\sign(ry_{\infty})) \prod_{j=1}^{d}W_{q_{j}\sign(r_{j}y_j)/2, \nu_{j}} (4\pi |r_{j}y_{j}|),
\end{equation}
%\footnote{$\chi_\infty =1$ pero eso no se ve en todas partes. Revisar.}
in light of \eqref{eq:Iwasawa} and \cite[\S 2.3.4]{BM09} and where $\epsilon_{\varpi}: \lbrace \pm 1 \rbrace ^{d} \rightarrow \lbrace \pm 1 \rbrace$ is a character depending only on the representation $\varpi$.
Here $W_{l,\nu}(y)$ denotes the $W$-Whittaker function
  %\begin{equation}\label{eq:d^r}
%d^{r}(q,\nu):= \frac{1}{\sqrt{2^{d}|D_{F}N(r)|}} \prod_{j=1}^{d} \frac{e^{\pi i q_{j}}}{\Gamma \left( \frac{1}{2}+\nu_{j}+ \frac{q_{j}}{2}\sign (r_{j}) \right)} % 
%\end{equation}
%the coefficient $c^{\ai,r}(\varpi)$ depends only on the representation $\varpi$ and not on  $\psi_{\varpi, q}$.
and  $c^{\ai,r}( f)$ is  the   Fourier coefficient  of  order $r$ of $ f_\ai$, for $r\in \ai^{-1} \mathfrak{d}^{-1}$.

%\footnote{{\red Creo que en el $d^r$ lo llamar\'ia $d^{\ai,r}$ y pondr\'ia $N(r\ai)$ en el denominador }}
% \footnote{AcÃ¡ se podrÃ­an  mencionar las propiedades (de d y W) que se usan en la prueba del Teor 5.3.}

%{\red By \eqref{eq:L2classical}, the Fourier terms of an adelic automorphic form will be indexed by a pair consisting of an ideal class and an element of $F$. Given $f \in L^{2}\big(\GL_2 (F)\ba \GL_2 (\mathbb{A}_{F}),\chi, \psi \big)_q$ we define $F_{\ai,r}f = F_r f_{\ai}$. }

%{\red \begin{remark} Note that the Fourier coefficient $c^{\ai,r}$ of an adelic automorphic form only depends on the ideal $\ai \langle r \rangle.$ \end{remark}}

 %Let $\psi: \A \rightarrow S^1$ be the unique continuous additive character, which is trivial on $F$,  it agrees with $\psi_{\infty}(x) = e^{2\pi i (x_1 + \ldots + x_d)}$ on $F_\infty$ and, on $F_{\p}$,  it is trivial on $\mathfrak{d}_{\p}^{-1}$ and non trivial on $\p^{-1}\mathfrak{d}_{\p}^{-1}$.\footnote{Check this.}
 %\footnote{Usar\'ia $\psi$ para este car\'acter y cambiar\'ia luego el car\'acter central ad\'elico que se llama $\psi$ tambi\'en}

Now let  $f\in \textrm{FS}_q$  be an adelic automorphic form of weight $q$ such that  $C_{j}f= (\frac{1}{4}-\nu_{j}^{2})f$ for $1 \le j \le d$, with $\nu_{j} \in \C$ for each $j$. 
%Then the Fourier terms $F_{r}f$ of $f$ are also eigenfunctions of the Casimir operators.
%The growth condition  implies that for $r \neq 0$ the Fourier term $F_{r}f$ is a multiple of  a $W$-Whittaker function. 
By the discussion above, every component $f_\ai$ of $f$ is a classical automorphic form of weight $q\in \Z^d$ and hence has a Fourier-Whittaker expansion
%When we restrict to the classical component $f_{\ai}$ of $f$ we have
\begin{align}\notag
f_{\ai} \left(\left( \begin{matrix}
y_\infty & x_\infty\\
0 & 1 \end{matrix} \right)k_\infty \right) &= \overline{\chi (\pi_{\mathfrak{a}})}
f  \left(\left(\begin{matrix}
\pi_{\ai} & 0\\
0 & 1 \end{matrix} \right)\left( \begin{matrix}
y_\infty & x_\infty\\
0 & 1 \end{matrix} \right)k_\infty \right)\\
& \label{eq: fourier expansion}=  \overline{\chi (\pi_{\mathfrak{a}})}\sum_{r \in \ai^{-1}\mathfrak d^{-1} } c^{\ai , r}(f) d^{\ai,r}(q,\nu) W_q(r,\nu;a(y_{\infty})) \psi(rx_\infty) \phi_q(k_\infty ),
\end{align}
for every $y_\infty \in F_\infty^\times$, $x \in F_\infty$ and $k_\infty \in K_\infty$.

\section{Action of the center} 
\label{section: grupo of clases}

In this section we recall the  action  of $Z(\mathbb{A}_F^\times)$ in the components of  $\GL_2(F) \ba \GL_2(\mathbb{A}_F) / K_{0}(\mathfrak{I}) Z(F_{\infty}^\times)$ 
 (see \cite[\S 6]{Ve04}).
%\footnote{Agregar $Z(F_{\infty})$?}

Let $g_{\infty} \in \GL_{2}(\R)^d$, let $\mathfrak{a}, \mathfrak{b}$ be   fractional ideals of $F$ and $\pi_{\mathfrak{a}}, \pi_{\mathfrak{b}}$ be  associated ideles. As a consequence of  %\eqref{strongapproxSL2}
strong  approximation in $\SL_2$ 
we have
\begin{align} \notag
\left(\begin{matrix}
\pi_{\mathfrak{b}} & 0\\
0 & \pi_{\mathfrak{b}}
\end{matrix} \right)\left(\begin{matrix}
\pi_{\mathfrak{a}} & 0\\
0 & 1
\end{matrix} \right)g_{\infty}& =
 %\left(\begin{matrix}
%\pi_{\mathfrak{b}}\pi_{\mathfrak{a}} & 0\\
%0 & \pi_{\mathfrak{b}}
%\end{matrix} \right)g_{\infty}
%\\ \notag &= 
\left(\begin{matrix}
\pi_{\mathfrak{b}}^{-1} & 0\\ 0 & \pi_{\mathfrak{b}}
\end{matrix} \right)\left(\begin{matrix}
\pi_{\mathfrak{b}}^{2}\pi_{\mathfrak{a}} & 0\\
0 & 1
\end{matrix} \right)g_{\infty}\\&
=  \label{componentes}i(\gamma^{-1})\left(\begin{matrix}
\pi_{\mathfrak{b}}^{2}\pi_{\mathfrak{a}} & 0\\
0 & 1
\end{matrix} \right)i_\infty(\gamma) g_{\infty}k_{\gamma}
\end{align}
for some $k_{\gamma} \in K_{0}(\mathfrak{I})$ and $\gamma \in \GL_2(F)$. % and $g'_{\infty} \in \GL_2 (\R)^d$. 
Furthermore we see that
%In \cite[sec 6]{Ve04} one shows that
\begin{equation} \label{flecha}
\gamma \in i(\GL_2 (F)) \cap \GL_{2}(\R)^d \left(\begin{smallmatrix}\pi_{\mathfrak{b}}^{2}\pi_{\ai} & 0 \\
0 & 1 \end{smallmatrix}\right)k_{\gamma} \left(\begin{smallmatrix}\pi_{\mathfrak{b}}\pi_{\ai} & 0 \\
0 & \pi_{\mathfrak{b}} \end{smallmatrix}\right)^{-1}, \textrm{ with } k_{\gamma} \in K_0 (\mathfrak{I}).
\end{equation}
Thus, \eqref{componentes} shows that translation by the  central element $\left(\begin{smallmatrix}
 \pi_{\mathfrak{b}} & 0\\
 0 & \pi_{\mathfrak{b}}
 \end{smallmatrix} \right)$  moves the component $\mathfrak{a}$ in \eqref{strongapprox} to the  component $\mathfrak{ab}^{2}$.

 Now, as in \cite{Ve04}, we  denote by 
$\Gamma(\ai \rightarrow \mathfrak{ab}^{2})$ the set of $ \gamma \in i(\GL_2 (F))$ satisfying $(\ref{flecha})$. %\rbrace$.
  Then, one checks that for any $\gamma_{\mathfrak{a}\rightarrow \mathfrak{ab}^{2}} \in \Gamma(\mathfrak{a}\rightarrow \mathfrak{ab}^{2})$, one has
%one can easily check that  
 $$\gamma_{\ai \rightarrow \mathfrak{ab}^{2}}\Gamma_{0}(\mathfrak{I},\ai ) = \Gamma_{0}(\mathfrak{I},\mathfrak{ab}^{2})\gamma_{\ai \rightarrow \mathfrak{ab}^{2}}=\Gamma(\mathfrak{a}\rightarrow \mathfrak{ab}^{2}).$$

% This action of the class groupwill be relevant in our computations later, so we will describe the elements %$\gamma$ that appear after action of the center  when we use strong approximation in  $\SL_2$.
%% The action of the class group on the connected  components of the space $\GL_{2}(F) \ba \GL_2 (\mathbb{A}_{F}) / K_{0}(\mathfrak{I})$  translates into an  action %% of the elements of the center  $Z(\mathbb{A}_{F,f})$ on  functions in $L^{2}\big(\GL_2 (F) \ba \GL_2 (\mathbb{A}_{F})/ K_{0}(\mathfrak{I})\big)$.

 Given $f\in \textrm {FS}$ with classical components ${f_{\ai}}_i \in L^2 (\Gamma_0(\mathfrak{I}, \mathfrak{a})\ba \GL_2(\R)^d, \chi^{-1})$, $i=1,\ldots,h_F$,
 %\footnote{{\olive saqu\'e los $K_0$}}
 and $z := \left( \begin{smallmatrix}
\pi_{\mathfrak{b}} & 0\\
0 & \pi_{\mathfrak{b}}
\end{smallmatrix} \right)$, using  \eqref{componentes} we compute  
%on  $f$. 
%We have that %as $\pi (z). f= f'$, where
%\begin{equation} \label{eq:zf}
%(z.f)_{\ai}(g_{\infty})=
%\chi_{_{f}} (\pi_{\mathfrak{b}})^{2} \chi_{_{f}}(k_{\gamma_{\ai \rightarrow \mathfrak{ab}^{2}}})f_{\mathfrak{ab}^{2}}(\gamma_{\ai \rightarrow \mathfrak{ab}^{2}}g_{\infty}).
%\end{equation}
%\begin{proof}
%Indeed,
\begin{align} \notag%\label{eq:zf}
(z.f)_{\ai}(g_{\infty})=[\pi(z).f]_{\ai}(g_{\infty})&= \overline{\chi(\pi_{\ai})}f\left(\left(\begin{smallmatrix}
\pi_{\ai} & 0\\\notag
0 & 1
\end{smallmatrix} \right)\left(\begin{smallmatrix}
\pi_{\mathfrak{b}} & 0\\
0 & \pi_{\mathfrak{b}}
\end{smallmatrix} \right)g_{\infty}\right) \\\notag &= \overline{\chi(\pi_{\ai})}f\left(\left(\begin{smallmatrix}
\pi_{{\mathfrak{b}}}^{2}\pi_{\ai} & 0\\
0 & 1
\end{smallmatrix} \right)\gamma_{\ai \rightarrow \mathfrak{ab}^{2}}g_{\infty}k_{\gamma_{\ai \rightarrow \mathfrak{ab}^{2}}}\right) \\
\notag & =\chi(\pi_{\mathfrak{b}}^{2}) \overline{\chi(\pi_{\ai}\pi_{\mathfrak{b}}^{2})}f\left(\left(\begin{smallmatrix}
\pi_{\mathfrak{b}}^{2}\pi_{\ai} & 0\\
0 & 1
\end{smallmatrix} \right)\gamma_{\ai \rightarrow \mathfrak{ab}^{2}}g_{\infty}\right) \chi_f(k_{\gamma_{\ai \rightarrow \mathfrak{ab}^{2}}}) \\ \label{eq:zf} &= \chi(\pi_{\mathfrak{b}}^{2}) f_{\mathfrak{ab}^{2}}(\gamma_{\ai \rightarrow \mathfrak{ab}^{2}}g_{\infty})\chi_f(k_{\gamma_{\ai \rightarrow \mathfrak{ab}^{2}}}).
\end{align}
%where %, in $^{(*)}$ 
%in the second equality we use \eqref{componentes}.%\footnote{{\red Esto vale para todas las $f$ ad\'elicas, si ya tomamos $f$ autofunci\'on del centro ad\'elico hay que agregar el car\'acter central.}}
%\end{proof}

We  will also need  the following  explicit description of the set $\Gamma (\ai  \rightarrow \mathfrak{ab}^{2})$. 

\begin{lemma}  \cite[(86)]{Ve04}
\begin{equation}\notag
\Gamma(\ai \rightarrow \mathfrak{ab}^{2}) = \left\lbrace \left( \begin{smallmatrix}
a & b\\
c & d
\end{smallmatrix} \right) : a\in \mathfrak{b}, b\in \mathfrak{ab}, c\in \ai^{-1}\mathfrak{b}^{-1}\mathfrak{I}, d\in \mathfrak{b}^{-1}, ad-bc \in \mathcal{O}_{F}^{\times} \right\rbrace .
\end{equation}
\end{lemma}

Let  
$\GL_{2}(F) = P_{F} \sqcup C_{F}$,  the Bruhat decomposition,  with
$C_{F}=\left\lbrace \left( \begin{matrix}
* & *\\
c \neq 0 & * 
\end{matrix} \right)\right\rbrace$, the big Bruhat cell.
%\begin{align*}
%P_{F}=\left\lbrace \left( \begin{matrix}
%* & *\\
%0 & * 
%\end{matrix} \right) \in \GL_{2}(F) \right\rbrace , \; %C_{F}=\left\lbrace \left( \begin{matrix}
%* & *\\
%c \neq 0 & * 
%\end{matrix} \right) \in \GL_{2}(F) \right\rbrace.
%\end{align*}
Then one has that   $\Gamma(\ai \rightarrow \mathfrak{ab}^{2})$ is contained in $C_{F}$, except when  $\mathfrak{b}$ is principal.

Indeed, if $\gamma \in \Gamma(\ai \rightarrow \mathfrak{ab}^{2}) \cap P_{F}$, then $\gamma = \left( \begin{smallmatrix}
a & b\\
0 & d
\end{smallmatrix} \right)$ with $ad \in \mathcal{O}_{F}^{\times}$, $a \in \mathfrak{b}$, $d \in \mathfrak{b}^{-1}$. Now,
% we are saying that the inverse of  
% $a$ is in $\mathfrak{b}^{-1}$, 
 since $ad= u\in \mathcal{O}_{F}^{\times}$, then $a^{-1}=du^{-1} \in \mathfrak{b}^{-1}$. Thus $a^{-1}\mathfrak{b} \subseteq \mathcal{O} \Rightarrow \mathfrak{b}\subseteq a\mathcal{O} \subseteq \mathfrak{b}$, that is $\mathfrak{b}= a\mathcal{O}$.

\section{Hecke operators}

In this section we define the Hecke operators on  adelic functions and then   translate the  action to the classical components  of the given function.

 Let $\mathfrak{p}$ be a prime ideal in $\mathcal{O}_{F}$ such that $\p \nmid \mathfrak{I}$, let $F_{\mathfrak{p}}$, $\mathcal{O}_{\mathfrak{p}}$ be the  completions of $F$ and $\mathcal{O}$ at $\mathfrak{p}$ respectively, and let $\pi_{\mathfrak{p}}$ be the uniformizer of the local ring $\mathcal{O}_{\mathfrak{p}}$.% {\red we have taken $i_f(\pi_\p)$ as the idele associated to the ideal $\p$.}\footnote{Puse esto mucho antes. Revisarlo.}

Let $\Delta (\mathfrak{p^{\ell}}) = \lbrace g \in M_{2}(\mathcal{O}_{\mathfrak{p}})  :  \operatorname{det}(g)\in \pi_{\mathfrak{p}}^{\ell}\mathcal{O}_{\mathfrak{p}} ^{\times
} \rbrace$.
Then 
\begin{align}
\Delta (\mathfrak{p^{\ell}})&= \GL_2 (\mathcal{O}_{\mathfrak{p}})\left( \begin{matrix}
\pi_{\mathfrak{p}}^{\ell} & 0\\
0 & 1
\end{matrix}\right)\GL_2 (\mathcal{O}_{\mathfrak{p}})
= \bigsqcup_{s = 0}^{\ell} \bigsqcup_{\beta \in \mathcal{O}_{\mathfrak{p}}/\pi_{\mathfrak{p}}^{\ell - s}\mathcal{O}_{\mathfrak{p}}} \left( \begin{matrix}
\pi_{\mathfrak{p}}^{\ell - s} & \beta\\
0 & \pi_{\p}^{s}
\end{matrix}\right)\GL_2 (\mathcal{O}_{\mathfrak{p}}).  \notag
\end{align}

Given  $f\in \textrm {FS}_q$ 
% \in L^{2} (\GL_{2} (F) \ba \GL_{2} (\mathbb{A}_{F}),\chi, \psi)$,
%$f$  $K_{\infty}$-finite and
%$f$ $Z_{\infty}$-finite %\footnote{No hace falta eigenfunction?.{\red Creo que esto se puede sacar, ya definimos ese espacio.}} 
and $g \in \GL_2(\mathbb{A}_F)$, 
the  Hecke operator $T_{\mathfrak{p}^\ell}$ is defined by
\begin{align}\label{Heckedefn}
(T_{\mathfrak{p}^{\ell}}f)(g)=& \int_{\GL_{2}(\mathbb{A}_{f})} f(gx)%{\chi_{_{f}}}
\chi_{_{\Delta(\mathfrak{p}^{\ell})}}(x) {dx} =
\sum_{s = 0}^{\ell}\sum_{\beta \in \mathcal{O}_{\mathfrak{p}}/\pi_{\mathfrak{p}}^{\ell - s} \mathcal{O}_{\mathfrak{p}} } f\left(g\left(\begin{matrix}
\pi_{\mathfrak{p}^{\ell - s}} & \beta\\
0 & \pi_{\mathfrak{p}^{s}}
\end{matrix} \right) \right) %\notag
\end{align}
where $\chi_{_{\Delta(\mathfrak{p}^{\ell})}}(x)$ denotes  the characteristic function of 
$\Delta(\mathfrak{p}^{\ell})$.

We now determine the  classical  components $(T_{\mathfrak{p^\ell}}f)_{\ai}$ of $T_{\mathfrak{p^\ell}}(f)$,  for $f\in\FS$.%  acting on such a function. 

We have that  
$\left(\begin{matrix}
\pi_{\mathfrak{a}} & 0\\
0 & 1
\end{matrix} \right)\left(\begin{matrix}
\pi_{\mathfrak{p}}^{\ell - s} & \beta\\
0 & \pi_{\p}^{s}
\end{matrix} \right) = \left(\begin{matrix}
\pi_{\mathfrak{p}}^{- s} & \pi_{\ai}\beta\\
0 & \pi_{\p}^{s}
\end{matrix} \right)\left(\begin{matrix}
\pi_{\mathfrak{a}}\pi_{\p}^{\ell} & 0\\
0 & 1
\end{matrix} \right)$, and furthermore, by strong approximation in $\SL_2$,   $\left(\begin{matrix}
\pi_{\mathfrak{p}}^{- s} & \pi_{\ai}\beta\\
0 & \pi_{\p}^{s}
\end{matrix} \right)= \gamma_{\beta,s} k_{s}$ with $\gamma_{\beta,s} \in \SL_{2}(F)$ and $k_{s} \in K_0(\mathfrak{I})$. Thus
%= i_{f}(\gamma_{\beta})k_{\gamma_{\beta}},$ 
% $k_{\gamma_{\beta}} \in K_{0}(\mathfrak{I}, \ai \p^{\ell})$ (by \eqref{strongapproxSL2})  
%and also, $\widetilde{k_{\gamma_{\beta}}} \in K_{0}(\mathfrak{I})$ is the conjugate of  $k_{\gamma_{\beta}}$ by $\left(\begin{matrix}
%\pi_{\mathfrak{a}}\pi_{\p}^{\ell} & 0\\
%0 & 1
%\end{matrix} \right)$. 
%\end{proposition}

\begin{align}%\label{tpl}  
\notag (T_{\mathfrak{p}^{\ell}}f)_{\mathfrak{a}}(g_{\infty})%&= 
%\sum_{s = 0}^{\ell}\sum_{\beta \in \mathcal{O}_{\mathfrak{p}}/ \pi_{\mathfrak{p}}^{\ell - s} \mathcal{O}_{\mathfrak{p}} } f_{\mathfrak{a}}\left(g_{\infty}\left(\begin{matrix}
%\pi_{\mathfrak{p}}^{\ell - s} & \beta\\0 & \pi_{\p}^{s}
%\end{matrix} \right) \right) \\ 
&= \sum_{s = 0}^{\ell}\sum_{\beta \in \mathcal{O}_{\mathfrak{p}}/\pi_{\mathfrak{p}}^{\ell - s} \mathcal{O}_{\mathfrak{p}} }\overline{\chi(\pi_{\ai})}f\left(\left(\begin{matrix}
\pi_{\mathfrak{a}} & 0\\
0 & 1
\end{matrix} \right)\left(\begin{matrix}
\pi_{\mathfrak{p}}^{\ell - s} & \beta\\
0 & \pi_{\p}^{s}
\end{matrix} \right)g_{\infty} \right)  \notag
\\&= \notag \sum_{s = 0}^{\ell}\sum_{\beta \in \mathcal{O}_{\mathfrak{p}}/  \pi_{\mathfrak{p}}^{\ell - s}\mathcal{O}_{\mathfrak{p}}}\overline{\chi(\pi_{\ai})}f\left(\left(\begin{matrix}
\pi_{\mathfrak{a}}\pi_{\p}^{\ell} & 0\\
0 & 1
\end{matrix} \right)i_{\infty}(\gamma_{\beta,s}^{-1})g_{\infty} \right)\chi_f(k_{{s}}) \\
&=\sum_{s = 0}^{\ell}\sum_{\beta \in \mathcal{O}_{\mathfrak{p}}/  \pi_{\mathfrak{p}}^{\ell - s}\mathcal{O}_{\mathfrak{p}}}\chi(\pi_{\p}^{\ell})f_{\ai \p^{\ell}}\left(i_{\infty}(\gamma_{\beta,s}^{-1})g_{\infty} \right)\chi_f (k_{s}) \label{eq: tpl en a}.
\end{align}
%%\footnote{{\olive Cambi\'e $\widetilde{k_{\gamma_\beta}}$ por $k_{\gamma_\beta}$ para alivianar la notaci\'on. Adem\'as confund\'ia lo de conjugar por $K_0(\mathfrak{I}, \ai)$.OK.R}}
%{\olive 

Thus, we have expressed   $(T_{\p^{\ell}}f)_{\ai}$ in terms of values of  $f_{\ai \p^{\ell}}$.    Now, we shall see that,  in light of  \eqref{componentes}, 
% in  Section~\ref{section: grupo of clases}, 
when $\mathfrak p^{\ell}$ is a square in the 
class group $\mathcal{C}_F$,  then   $\left(T_{\p^{\ell}}\left(\begin{smallmatrix}
\pi_{\mathfrak{b}}^\ell & 0\\
0 & \pi_{\mathfrak{b}}^{\ell}
\end{smallmatrix}\right)f\right)_{\ai}$ can be expressed in terms of values of $f_{\ai}$, again at the initial component $\ai$.  
%{\olive Sacar\'ia esta oraci\'on que sigue.} We note that  $\p^{\ell}$ is a square in the class group, for instance if $h$ is odd, so that  $\p$ is a square, and also if $\ell$ is even.

% \footnote{ Agregue en el prÃ³ximo enunciado ese cambio con $\p$ un cuadrado, pero eso traerÃ¡ otros cambios mÃ¡s adelante. Los podes hacer vos?}
%{\olive In the following we will need a prime ideal $\p$ be an square in $\mathcal{C}^{+}$.}
%{\red 
%	\begin{remark} \label{rk:generadores} If $\p$ is a square in $\mathcal{C}^+$ there exists an ideal $\mathfrak{b}$ such that $\p\mathfrak{b}^2 = \langle \eta_{\p\mathfrak{b}^2} \rangle$ and $\eta_{\p\mathfrak{b}^2}$ is totally positive and also, given any $\ell\in \Z$ we have that %and an ideal $\widetilde{\mathfrak{b}}$ such that $\p^{\ell}\widetilde{\mathfrak{b}}^2 = \langle \eta_{\p^{\ell}\widetilde{b}^2} \rangle$, with $\eta_{\p^{\ell}\widetilde{b}^2}$ totally positive,  then  
%	 $\p^{\ell} \mathfrak{b}^{2\ell} = \langle  \eta_{\p\mathfrak{b}^2}^{\ell}\rangle $.
%	 = \langle \eta_{\p^{\ell}b^{2\ell}}  \rangle $, and   $\eta_{\p\mathfrak{b}^2}^{\ell}$ and $\eta_{\p^{\ell}b^{2\ell}} $ differ by a totally positive unit. 
%\footnote{Check this! Probablemente mejor ponerlo en otro lado.}
 %\end{remark} 
\begin{proposition} \label{relacion de los tpl} 
	Let $f\in \FS_{q}$, let $\p\subset \mathcal O_F$ be a prime ideal, $\p \nmid \mathfrak{I}$  and $\p$ a square in $C_F$. Let  $\mathfrak{b}$ be an integral ideal such that $\p\mathfrak{b}^{2} =  \eta_{\p\mathfrak{b}^2} \mathcal O_F$ with $\eta_{\p\mathfrak{b}^2} \in \mathcal{O}_F$ totally positive.
%and $f \in L^{2} (\GL_{2} (F) \ba \GL_{2} (\mathbb{A}_{F}),\chi, \omega)$
Then   %in the notation of this section
	\begin{equation}
(T_{\p^{\ell}}f)_{\ai}(g_{\infty}) =  \chi^{-1}(\pi_{\mathfrak{b}^{\ell}}) \sum_{s = 0}^{\ell} \chi_f(k_{s})  f_{\mathfrak a,s} (g_\infty),\,\, \textrm{where }
	\end{equation}
	\begin{equation}\label{eq:fs}
	f_{\mathfrak a,s} (g_\infty)=\sum_{\alpha_s \in \ai \p^{- s} / \ai\p^{\ell - 2s}} f_{\ai}  
	\left( \left(\begin{matrix}
\eta_{\p \mathfrak{b}^{2}}^{-\ell}a_{s} & \eta_{\p \mathfrak{b}^{2}}^{-\ell}\widetilde{b_{s}}\\
0 & a_{s}^{-1} 
\end{matrix} \right)  
	\left(\begin{matrix}
	1 & \alpha_s\\
	0 & 1
	\end{matrix}\right) g_{\infty} \right),
	\end{equation}
where $\widetilde{b_{s}} \in \mathcal{O}_F$ and	$a_s \in \p^s \mathfrak{b}^\ell$ satisfies $v_{\p}(a_s) = s + \ell v_{\p}(\mathfrak{b})$, for any $0 \le s \le \ell$. 
\end{proposition}
%\footnote{Include the properties  to be used later of $a_s$ and $b_s$?}

\begin{proof} 
 We have that 
	$\left(\left[\pi_{\mathfrak{b}}^{\ell}\right] T_{\p^{\ell}} f\right)_{\ai}= \chi (\pi_{\mathfrak b}^\ell)(T_{\p^{\ell}} f)_{\ai}$, 
	hence we must compute
	\begin{align}
	\chi (\pi_{\mathfrak b}^\ell)(T_{\p^{\ell}} f)_{\ai}(g_{\infty}) &= %{\red \psi^{-1}(\pi_{\mathfrak{b}^{\ell}}) }
	\sum_{s = 0}^{\ell}\sum_{\beta \in \mathcal{O}_{\mathfrak{p}}/  \pi_{\mathfrak{p}}^{\ell - s}\mathcal{O}_{\mathfrak{p}}}\!\!\!\!\!\overline{\chi(\pi_{\ai})}f\left( \left(\begin{matrix}
	\pi_{\ai} & 0\\
	0 & 1
	\end{matrix}\right)\left(\begin{matrix}
	\pi_{\p}^{\ell - s} & \beta \\
	0 & \pi_{\p}^{s}
	\end{matrix}\right)\left(\begin{matrix}
	\pi_{\mathfrak{b}}^{\ell} & 0\\
	0 & \pi_{\mathfrak{b}}^{\ell}
	\end{matrix}\right)g_{\infty}\right).
	\end{align}
% We  need to   analyze   certain  matrices in $\GL_{2}(F)$ that  appear by strong  approximation.	
Now 
\begin{align}\notag
\left(\begin{matrix}
\pi_{\ai} & 0\\
0 & 1
\end{matrix}\right)\left(\begin{matrix}
\pi_{\p}^{\ell - s} & \beta \\
0 & \pi_{\p}^{s}
\end{matrix}\right)\left(\begin{matrix}
\pi_{\mathfrak{b}}^\ell & 0\\
0 & \pi_{\mathfrak{b}}^{\ell}
\end{matrix}\right)
&=\left(\begin{matrix}
\pi_{\mathfrak{b}}^{-\ell}\pi_{\p}^{- s} & \pi_{\ai}\pi_{\mathfrak{b}}^{\ell}\beta\\
0 & \pi_{\p}^{s}\pi_{\mathfrak{b}}^{\ell}
\end{matrix}\right) \left(\begin{matrix}
\pi_{\mathfrak{b}}^{2\ell}\pi_{\p}^{\ell}\pi_{\ai} & 0\\
0 & 1
\end{matrix}\right)\\
&= \notag  i(\gamma_{\beta,s})\left(\begin{matrix}
\pi_{\mathfrak{b}}^{2\ell}\pi_{\p}^{\ell}\pi_{\ai} & 0\\
0 & 1
\end{matrix}\right) i_{\infty}(\gamma_{\beta,s}^{-1}) {k_{s}} 
\\&=  
 i(\gamma_{\beta,s})i\left(\left(\begin{matrix}
\eta_{\p\mathfrak{b}^{2}}^{\ell} & 0\\
0 & 1
\end{matrix}\right)\right)\left(\begin{matrix}
\pi_{\ai} & 0\\
0 & 1
\end{matrix}\right)i_{\infty}\left(\left(\begin{matrix}
\eta_{\p\mathfrak{b}^{2}}^{\ell} & 0\\
0 & 1
\end{matrix}\right)^{-1}\gamma_{\beta,s}^{-1}\right)k_{s}.\label{eq: st app}
\end{align}

We have used \eqref{strongapproxSL2} for  $\left(\begin{matrix}
\pi_{\mathfrak{b}}^{-\ell}\pi_{\p}^{- s} & \pi_{\ai}\pi_{\mathfrak{b}}^{\ell}\beta\\
0 & \pi_{\p}^{s}\pi_{\mathfrak{b}}^{\ell}
\end{matrix}\right)$ with $\gamma_{\beta,s} \in \SL_{2}(F)$ and  $k_{s}\in K_{0}(\mathfrak{I})$,
and also the fact that $\mathfrak{b}^{2\ell}\p^{\ell} = \eta_{\p\mathfrak{b}^{2}}^{\ell}\mathcal{O}_F $. We  pick 
%$\eta_{\p\mathfrak{b}^2}$ so that $\pi_{\mathfrak{b}}^{2}\pi_{\p} = 
$i_f( \eta_{\p\mathfrak{b}^{2}}^{\ell})$ as an idele associated to $\mathfrak{b}^{2\ell}\p^{\ell}$.  
%Then, for any $\ell$,  $\pi_{\mathfrak{b}}^{2\ell}\pi_{\p}^{\ell} = i_f( \eta_{\p\mathfrak{b}^{2}}^{\ell})$.  
 
Now we look carefully at the element  $\gamma_{\beta,s}^{-1} \in \SL_{2}(F)$  appearing in \eqref{eq: st app}.
 We distinguish two cases: when $\p^{s}\mathfrak{b}^{\ell}$ is a principal ideal and  when it is not.
\

If $s \in \lbrace 0, 1, \ldots, \ell \rbrace$ is such that  $\p^{s}\mathfrak{b}^{\ell}$  \emph{is principal}, then we can take $\gamma_{\beta, s}$ a diagonal matrix
\begin{equation}\label{eq: matrices diagonales}
\left( \begin{matrix}
\pi_{\mathfrak{b}}^{-\ell}\pi_{\p}^{-s} & \pi_{\ai}\pi_{\mathfrak{b}}^{\ell}\beta \\
0 & \pi_{\p}^{s}\pi_{\mathfrak{b}}^{\ell}
\end{matrix} \right) = i_f\left(\left(\begin{matrix}
a_{s}^{-1} & 0\\
0 & a_{s}
\end{matrix} \right)\right) \left( \begin{matrix}
1 & a_{s} \pi_{\mathfrak{a}}\pi_{\mathfrak{b}}^{\ell}\beta\\
0 & 1
\end{matrix}\right)
\end{equation}
where $a_{s} \in \mathcal{O}_{F}$ is such that $ a_{s}\mathcal O_F  = \p^{s}\mathfrak{b}^{\ell}$. That is,  if $\p^{s}\mathfrak{b}^{\ell}$ is principal, we have that \begin{align}\label{eq: diagonal para s principal}
&\gamma_{\beta,s}^{-1} = \left( \begin{matrix}
a_{s} & 0\\
0 & a_{s}^{-1}
\end{matrix}\right)\in \SL_{2}(F) \;\;\; \textrm{ and } \;\;\;\;
%\label{eq: compacto para s principal}
k_{s} = \left( \begin{matrix}
1 & a_{s} \pi_{\mathfrak{a}}\pi_{\mathfrak{b}}^{\ell}\beta\\
0 & 1
\end{matrix}\right) \in K_{0}(\mathfrak{I}).
\end{align}

In the case when  $\p^{s}\mathfrak{b}^{\ell}$ \emph{is not principal},  % Starting from   \eqref{gammabetas} we have
%\begin{align}
%i_f(\gamma_{\beta,s}^{-1}) \in i_{f}(\GL_{2}(F)) \cap \left(\begin{matrix}
%\pi_{\mathfrak{b}}^{2}\pi_{\p}^{\ell}\pi_{\ai} & 0\\
%0 & 1
%\end{matrix} \right)K_{0}(\mathfrak{I}) \left(\begin{matrix}
%\pi_{\mathfrak{b}}^{-1}\pi_{\p}^{s-\ell}\pi_{\ai}^{-1} & -\beta %\pi_{\mathfrak{b}}^{-1}\pi_{\p}^{-\ell}\\
%0 & \pi_{\p}^{-s}\pi_{\mathfrak{b}}^{-1}
%\end{matrix} \right).
%\end{align}
let $\left(\begin{smallmatrix}
a & b\\
c & d
\end{smallmatrix} \right) \in K_{0}(\mathfrak{I})$. 
By \eqref{flecha} applied to $\gamma_{\beta,s}$ we have
\begin{align*}
&\left(\begin{matrix}
\pi_{\mathfrak{b}}^{2\ell}\pi_{\p}^{\ell}\pi_{\ai} & 0\\
0 & 1
\end{matrix} \right) 
\left(\begin{matrix}
a & b\\
c & d
\end{matrix} \right) 
\left(\begin{matrix}
\pi_{\mathfrak{b}}^{-\ell}\pi_{\p}^{s-\ell}\pi_{\ai}^{-1} & -\beta \pi_{\mathfrak{b}}^{-\ell}\pi_{\p}^{-\ell}\\
0 & \pi_{\p}^{- s}\pi_{\mathfrak{b}}^{-\ell} \end{matrix} \right)
=\left(\begin{matrix}
\pi_{\mathfrak{b}}^{\ell}\pi_{\p}^{s}a & \pi_{\mathfrak{b}}^{\ell}\pi_{\p}^{\ell - s}\pi_{\ai}b\\
\pi_{\ai}^{-1}\pi_{\mathfrak{b}}^{-\ell}\pi_{\p}^{s - \ell} c & \pi_{\p}^{- s}\pi_{\mathfrak{b}}^{-\ell} d
\end{matrix} \right)\left(\begin{matrix}
1 & -\beta \pi_{\p}^{- s}\pi_{\ai}\\
0 & 1 \end{matrix} \right) .
\end{align*}
and when we intersect with $i_{f}(\GL_{2}(F))$ we have  
%\begin{align*}\notag
%\left(\begin{matrix}
%\pi_{\mathfrak{b}}\pi_{\p}^{s}a & \pi_{\mathfrak{b}}\pi_{\p}^{\ell - s}\pi_{\ai}b\\
%\pi_{\ai}^{-1}\pi_{\mathfrak{b}}^{-1}\pi_{\p}^{s - \ell} c & \pi_{\p}^{ - s}\pi_{\mathfrak{b}}^{-1}d
%\end{matrix} \right) \cap i_{f}(\GL_{2}(F)) \in i_f \left(\begin{matrix}
%\p^{s}\mathfrak{b} & \mathfrak{b}\p^{\ell - s} \ai\\
%\ai^{-1}\mathfrak{b}^{-1}\p^{s - \ell} \mathfrak{I} & \p^{- s}\mathfrak{b}^{-1}
%\end{matrix} \right). 
%\end{align*}
%\textcolor{red}{decir simplemente que est\'a en el gamma flecha}Since the   determinant is a unit, we obtain 
\begin{align*}
\left(\begin{matrix}
\pi_{\mathfrak{b}}^{\ell}\pi_{\p}^{s}a & \pi_{\mathfrak{b}}^{\ell}\pi_{\p}^{\ell - s}\pi_{\ai}b\\
\pi_{\ai}^{-1}\pi_{\mathfrak{b}}^{-\ell}\pi_{\p}^{s - \ell} c & \pi_{\p}^{- s}\pi_{\mathfrak{b}}^{-\ell}d
\end{matrix} \right) \cap i_{f}(\GL_{2}(F)) \in i_f \left( \Gamma ( \ai \p^{ \ell - 2s} \mapsto \ai \p^{ \ell} \mathfrak{b}^{2\ell})\right).
\end{align*}
On the other hand,
\begin{align*}
\left(\begin{matrix}
1 & -\beta \pi_{\p}^{- s}\pi_{\ai}\\
0 & 1 \end{matrix} \right)  \cap i_{f}(\GL_{2}(F)) \in \Gamma_{N} (\mathfrak{I}, \ai \p^{- s}/\ai \p^{ \ell - 2s}),
\end{align*}
since $\beta \in \mathcal{O}_{\p}/ \pi_{\p}^{\ell - s}\mathcal{O}_{\mathfrak{p}}$ and  then $\beta \pi_{\p}^{- s}\pi_{\ai} \cap i(F) \in i_f\left(\ai \p^{- s}/\ai \p^{ \ell - 2s} \right)$. 
We thus have a decomposition
%\footnote{Yo cambiarÃ­a $\gamma_{\alpha,s}$ por $\gamma'_{\beta,s}$. No hay dependencia de alpha, sino que alpha depende de beta y de s, no?}\footnote{{\olive Notar que escrib\'i directamente la matriz en lugar de usar $\gamma$.}}
 \begin{equation}\label{gamma beta s}  i_{\infty}(\gamma_{\beta,s}^{-1}) =  \left(\begin{matrix}
a_s & b_s\\
c_s & d_s
\end{matrix} \right)  \left( \begin{matrix}
1 & \alpha_s\\
0 & 1
\end{matrix}\right),
\end{equation}
where % \textcolor{red}{ya dijimos donde est\'a el gamma}
%\begin{equation}\label{gammaalpha}
%$\gamma_{\alpha, s} =
$ \left(\begin{smallmatrix}
a_s & b_s\\
c_s & d_s
\end{smallmatrix} \right) \in i_f \left( \Gamma ( \ai \p^{ \ell - 2s} \mapsto \ai \p^{ \ell} \mathfrak{b}^{2\ell})\right)$,
%\left(\begin{matrix}
%\p^{s}\mathfrak{b} & \mathfrak{b}\p^{\ell - s} \ai\\
%\ai^{-1}\mathfrak{b}^{-1}\p^{s - \ell}\mathfrak{I} & \p^{- s} \mathfrak{b}^{-1}$
%\end{matrix} \right) , 
%\end{equation} 
$ a_sd_s-b_sc_s = u_s \in \mathcal{O}_{F}^{\times}$, 
and $\alpha_s \in \ai \p^{- s}/\ai \p^{ \ell - 2s}$.  Furthermore, by Lemma 3.1, $a_s \in \p^{s}\mathfrak{b}^{\ell}\subset \mathcal O_F$ and we may assume, by modifying $d_s$ if necessary, that   
%Note that by strong approximation we have 
\begin{align}\label{a_s, b_s}
&v_{\p}(a_s) = s + \ell v_{\p}(\mathfrak{b}).
%\textrm{ and } v_{\mathfrak{q}}(a_s) = v_{\mathfrak{q}}(\mathfrak{b}) \textrm{ for all primes } \mathfrak{q} \neq \p \textrm{ {\olive dividing $\mathfrak{b}.$}}
\end{align}
%\footnote{C'omo sale esto? Y si q no divide  b?}
%Antes de reemplazar esto en la Proposition~\ref{prop:tpl} 
%We may say something more about  \eqref{gammaalpha}. 

Now we claim that we can write %if
%\begin{lemma}\label{lemma:gammar} 
%\footnote{Revisar algunas cuentas del Lema}  
% $\gamma_{\alpha, s}$ is as in \eqref{gammaalpha},  then
%% {\olive 
$\left(\begin{smallmatrix}
a_s & b_s\\
c_s & d_s
\end{smallmatrix} \right)$ as a product of a matrix in $\Gamma_0(\mathfrak{I}, \ai\p^{\ell}\mathfrak{b}^{2\ell})$ and an upper triangular matrix.
%\begin{equation}\label{eq: gamma_alpha}
%\gamma_{\alpha, s} = \widetilde{\gamma_{\alpha, s}}p_{\alpha,s}
%\end{equation}
%where $\widetilde{\gamma_{\alpha, s}} \in \Gamma_{0}(\mathfrak{I}, \ai \p^{\ell}\mathfrak{b}^{2})$ and $p_{\alpha,s}$ is a parabolic element.
%\end{lemma}
%\begin{proof}
Indeed, if $c_s = 0$ there is nothing to prove. If   $c_s \neq 0$, then $a_s \neq 0$,  since if $a_s = 0$ then   $b_sc_s = u$,  contradicting that $b_sc_s \in \mathfrak{I}$. 
As $a_s \in \p^{s}\mathfrak{b}^{\ell}$ and $c_s \in \ai^{-1}\mathfrak{b}^{-\ell}\p^{s - \ell}\mathfrak{I}$, there exists  $i \in \mathfrak{I}$ such that $ia_sc_s^{-1} \in \ai \p^{\ell} \mathfrak{b}^{2\ell}$. Thus we may decompose 
\begin{align} \label{eq:as}
 \left(\begin{matrix}
a_s & b_s\\
c_s & d_s
\end{matrix} \right) = \left(\begin{matrix}
1 & a_sc_s^{-1}i\\
c_sa_s^{-1} & i+u_s
\end{matrix} \right)\left(\begin{matrix}
a_s & -ic_s^{-1} + b_s\\
0 & a_s^{-1}
\end{matrix} \right).
\end{align} 
Now since $a_sc_s^{-1}i \in \ai \p^{\ell}\mathfrak{b}^{2\ell}$, then $c_sa_s^{-1} \in \ai^{-1} \p^{-\ell}\mathfrak{b}^{-2\ell}\mathfrak{I}$, and since $i_s + u_s \in \mathcal{O}_{F}$, we have that $\widetilde{\gamma_s}:=\left(\begin{smallmatrix}
1 & a_sc_s^{-1}i\\
c_sa_s^{-1} & i_s+u_s
\end{smallmatrix} \right) \in \Gamma_{0}(\mathfrak{I}, \ai \p^{\ell}\mathfrak{b}^{2\ell})$.
%\end{proof}
%\begin{proof}
%{\olive 

Substituting  \eqref{eq: st app} in the argument of $(\left[\pi_{\mathfrak{b}}^{\ell}\right] T_{\p^{\ell}} f)_{\ai}= \chi (\pi_{\mathfrak b}^\ell)(T_{\p^{\ell}} f)_{\ai}$, %equation \eqref{eq: tpl en a}, 
 using   \eqref{eq:zf}  and in light of \eqref{eq: diagonal para s principal}, \eqref{gamma beta s} and \eqref{eq:as}, %Lema~\ref{lemma:gammar}
 we obtain that 
 \begin{align*}
 \chi (\pi_{\mathfrak b}^\ell)(T_{\p^{\ell}} f)_{\ai}(g_\infty) &= %{\red \psi^{-1}(\pi_{\mathfrak{b}^{\ell}}) }
 \sum_{s = 0}^{\ell}\sum_{\beta \in \mathcal{O}_{\mathfrak{p}}/  \pi_{\mathfrak{p}}^{\ell - s}\mathcal{O}_{\mathfrak{p}}}\!\!\!\!\!\overline{\chi(\pi_{\ai})}f\left( \left(\begin{matrix}
\pi_{\ai} & 0\\
0 & 1
\end{matrix}\right)\left(\begin{matrix}
\pi_{\p}^{\ell - s} & \beta \\
0 & \pi_{\p}^{s}
\end{matrix}\right)\left(\begin{matrix}
\pi_{\mathfrak{b}}^{\ell} & 0\\
0 & \pi_{\mathfrak{b}}^{\ell}
\end{matrix}\right)g_{\infty}\right) \\
&= %{\red \psi^{-1}(\pi_{\mathfrak{b}^{\ell}}) } 
\sum_{s = 0}^{\ell}\sum_{\beta \in \mathcal{O}_{\mathfrak{p}}/  \pi_{\mathfrak{p}}^{\ell - s}\mathcal{O}_{\mathfrak{p}}}\!\!\!\!\overline{\chi(\pi_{\ai})}f\left( \left(\begin{matrix}
\pi_{\ai} & 0\\
0 & 1
\end{matrix}\right)i_{\infty}\left(\gamma_{\beta,s}\left(\begin{matrix}
\eta_{\p\mathfrak{b}^{2}}^{\ell} & 0\\
0 & 1
\end{matrix}\right)\!\!\right)^{-1}g_{\infty} \right)\chi_f(k_{s})\\
&= %{\red \psi^{-1}(\pi_{\mathfrak{b}^{\ell}}) } 
\sum_{s = 0}^{\ell}\sum_{\beta \in \mathcal{O}_{\mathfrak{p}}/  \pi_{\mathfrak{p}}^{\ell - s}\mathcal{O}_{\mathfrak{p}}}\!\!\!\!f_{\ai}\left(i_{\infty}\left(\gamma_{\beta,s}\left(\begin{matrix}
\eta_{\p\mathfrak{b}^{2}}^{\ell} & 0\\
0 & 1
\end{matrix}\right)\!\!\right)^{-1}g_{\infty} \right)\chi_f(k_{s})\\
&= %{\red \psi^{-1}(\pi_{\mathfrak{b}^{\ell}}) } 
\sum_{s = 0}^{\ell} \! \sum_{\alpha_s \in \ai \p^{- s} / \ai\p^{\ell - 2s}}\!\!\!\!\!\!\! f_{\ai}\left( i_{\infty}\left(\!\!\left(\begin{matrix}
\eta_{\p\mathfrak{b}^{2}}^{\ell} & 0\\
0 & 1
\end{matrix}\right)\!\!\right)^{-1}\!\!\!\!\!\widetilde{\gamma_s} \left(\begin{matrix}
a_s & -ic_s^{-1} + b_s\\
0 & a_s^{-1}
\end{matrix} \right) \left(\begin{matrix}
1 & \alpha_s\\
0 & 1
\end{matrix}\right) g_{\infty} \right)\chi_f(k_{s})\\
& = %{\red \psi^{-1}(\pi_{\mathfrak{b}^{\ell}}) } 
\sum_{s = 0}^{\ell} \sum_{\alpha_s \in \ai \p^{- s} / \ai\p^{\ell - 2s}}\!\!\! f_{\ai}\left(  \left(\begin{matrix}
\eta_{\p\mathfrak{b}^{2}}^{-\ell}a_s & \eta_{\p\mathfrak{b}^{2}}^{-\ell}(-ic_s^{-1} + b_s)\\
0 & a_s^{-1}
\end{matrix} \right) \left(\begin{matrix}
1 & \alpha_s\\
0 & 1
\end{matrix}\right) g_{\infty} \right)\chi_f(k_{s}), 
\end{align*}
 Now setting $\widetilde b_s = -ic_s^{-1} + b_s \in \mathcal O_F$, the proposition follows.

\end{proof}

 \section{Fourier terms and  Hecke operators}
The  goal of this section is to prove  results on the Fourier terms of the components of $T_{\mathfrak{p}^\ell} f$.
%, where $T_{\mathfrak{p}^\ell}$ is a Hecke operator and $f$ is an adelic automorphic form.

\begin{proposition}\label{Fourier terms}
In the notation of Proposition~\ref{relacion de los tpl}, the function $f_{\mathfrak a,s}$ in \eqref{eq:fs} is left $\Gamma(\mathfrak{I}, \ai)_{N}$--invariant and its Fourier   terms are given by
\begin{equation}\notag
F_{\ai,r}f_{\mathfrak a,s}\!(g_{\infty})\! = \!\!\!\! \begin{array}{ll}
 \left\lbrace \!\!\! 
    \begin{array}{ll}
      \frac{N(\p)^{s} N(\eta_{\p \mathfrak{b}^{2}}^{\ell})}{N(a_{s}^{2})} F_{\ai, {r\eta_{\p\mathfrak{b}^{2}}^{\ell}}/{a_{s}^{2}}}\,f_{\ai}\left(\!\!\left( \begin{matrix}
      a_{s}\eta_{\p\mathfrak{b}^{2}}^{-\ell} & \widetilde{b_{s}}\eta_{\p\mathfrak{b}^{2}}^{-\ell}\\
      0 & a_{s}^{-1}
      \end{matrix}\right)\!\!g_{\infty}\!\!\right) \textrm{if } r \in \p^{s}\ai^{-1}\mathfrak{d}^{-1},  \\ 0  \; \textrm{ otherwise}.
    \end{array}\right.
\end{array}%\footnote{Aclarar el k en el chi. Tratar de poner todo lo que se pueda, incluyendo a los eta, como potencias de $\eta_{pb^2}$. }
\end{equation} 
%where $a_s$, $\widetilde{b_s}$ and $\eta_{\p\mathfrak{b}^2}$ are as in Proposition~\ref{relacion de los tpl}.
 where as in  \eqref{eq:fs}, $\widetilde{b_{s}} \in \mathcal{O}_F$ and $a_s \in \p^s \mathfrak{b}^\ell$ satisfies $v_{\p}(a_s) = s + \ell v_{\p}(\mathfrak{b})$, for any $\,0 \le s \le \ell$. 
Furthermore, $ \p\mathfrak{b}^2 = \eta_{\p\mathfrak{b}^2}\mathcal{O}_F $ with $\eta_{\p\mathfrak{b}^2}$ totally positive.
% is such that $\pi_{\eta_{\p\mathfrak{b}^2}} = \pi_{\p}\pi_{\mathfrak{b}}^2$.\footnote{Revisar esto!.} 
\end{proposition}
%\footnote{\olive agregar $\chi(k)$ en la prop}
\begin{proof}
 The set  
$$\bigsqcup_{\alpha_s\in \ai \p^{-s}/\ai\p^{\ell - 2s}} \Gamma_{0}(\mathfrak{I}, \ai)\left( \begin{matrix}
      a_{s}\eta_{\p\mathfrak{b}^{2}}^{-\ell} & \widetilde{b_{s}}\eta_{\p\mathfrak{b}^{2}}^{-\ell}\\
      0 & a_{s}^{-1}
      \end{matrix}\right)\left( \begin{matrix}
      1 & \alpha_s\\
      0 & 1
      \end{matrix}\right)$$ 
is right  invariant by the  group  $\Gamma_{N}(\mathfrak{I}, \ai\p^{-s}) = \lbrace n(\alpha)  :  \alpha \in \ai \p^{-s} \rbrace$ and this  group  contains  $\Gamma_{N}(\mathfrak{I}, \ai)= \lbrace n(\alpha) \; : \; \alpha \in \ai \rbrace$ since $ \ai \subseteq \ai\p^{-s}$. This implies that   $f_{\mathfrak a,s}$ is left $\Gamma_{N}(\mathfrak{I}, \ai)$--invariant.

We now determine the   Fourier expansion of $f_{\mathfrak a,s}$. For any  ideal $\mathfrak{c}$ of $F$, set $v(\mathfrak{c}) := \Vol (\R^{d}/ \mathfrak{c})$. 
\begin{align*}
F_{\ai,r}f_{\mathfrak a,s}(g_{\infty})  :&= \frac{1}{v(\ai)} \int_{\R^{d}/\ai} e^{-2\pi i S(rx)} f_{\mathfrak a,s}(n(x)g_{\infty})dx \\
&=  \frac{1}{v(\ai)} \int_{\R^{d}/\ai\p^{-s}} \sum_{\alpha'_s\in \ai \p^{-s}/\ai}  e^{-2\pi i S(r(x + \alpha'_s))} f_{\mathfrak a,s}(n(x+ \alpha'_s)g_{\infty})dx \\
&= \frac{1}{v(\ai)}\bigg( \sum_{\alpha'_s\in \ai\p^{-s}/\ai}  e^{-2\pi i S(r\alpha'_s)}\bigg) \int_{\R^{d}/\ai\p^{-s}}\!\!\!\!\!\! e^{-2\pi i S(rx)}%\!\!\!\!\!\!  %\sum_{\alpha_s\in \ai\p^{-s}/\ai}   
f_{\mathfrak a,s}(n(x)g_{\infty})dx. 
\end{align*} 
%\footnote{Corrections made before.}
\
Now  \begin{equation} \notag
 \sum_{\alpha'_s\in \ai\p^{-s}/\ai}  e^{-2\pi i S(r \alpha'_s)} = \begin{array}{ll}
 \Bigg\lbrace \! 
    \begin{array}{ll}
      N( \p^{s}) \textrm{ if } r \in \p^{s}\ai ^{-1} \mathfrak{d}^{-1}\,, \\ 0  \; \textrm{ otherwise}.
    \end{array}
\end{array}.
\end{equation} 

%\

%and furthermore $\sum_{\alpha_s\in \ai\p^{-s}/\ai}   f_{\mathfrak a,s}(n(x + \alpha_s)g_{\infty})$ is $\Gamma_{N}(\mathfrak{I}, \ai \p^{-s})$-invariant. 

Thus, if we let  $r \in \p^{s}  \ai ^{-1}\mathfrak{d}^{-1} $, then
\begin{align*}
F_{\ai,r}f_{\mathfrak a,s}(g_{\infty})  :&= \frac{N(\p)^{s} }{v(\ai)} \int_{\R^{d}/\ai\p^{ - s}}\!\!\!\!\!\!\! e^{-2\pi i S(rx)}   f_{\mathfrak a,s}(n(x )g_{\infty})dx  \\
&= \frac{N(\p)^{s} }{v(\ai)} \int_{\R^{d}/\ai\p^{ - s}}\!\!\!\!\!\!\!\!\!\!\! e^{-2\pi i S(rx)}\!\!\! \sum_{\alpha_s \in \ai\p^{- s}/\ai\p^{\ell - 2s}} f_{\ai}\left( \left(\begin{matrix}
\eta_{\p \mathfrak{b}^{2}}^{-\ell}a_{s} & \eta_{\p \mathfrak{b}^{2}}^{-\ell} \widetilde{b_{s}}\\
0 & a_{s}^{-1} 
\end{matrix} \right)  \left(\begin{matrix}
1 & \alpha_s\\
0 & 1
\end{matrix}\right)n(x) g_{\infty} \right)dx \\
&= \frac{N(\p)^{s} }{v(\ai)} \int_{\R^{d}/\ai\p^{\ell - 2s}}e^{-2\pi i S(rx)}f_{\ai}\left( \left(\begin{matrix}
\eta_{\p \mathfrak{b}^{2}}^{-\ell}a_{s} & \eta_{\p \mathfrak{b}^{2}}^{-\ell} \widetilde{b_{s}}\\
0 & a_{s}^{-1} 
\end{matrix} \right)n(x) g_{\infty} \right)dx. \\
\end{align*}
Since 
%\begin{align*}
$\left(\begin{smallmatrix}
\eta_{\p \mathfrak{b}^{2}}^{-\ell}a_{s} & \eta_{\p \mathfrak{b}^{2}}^{-\ell}\widetilde{b_{s}}\\
0 & a_{s}^{-1} 
\end{smallmatrix} \right)\left( \begin{matrix}
1 & x \\
0 & 1 
\end{matrix} \right) = \left( \begin{smallmatrix}
1 & x\eta_{\p \mathfrak{b}^{2}}^{-\ell}a_{s}^{2} \\
0 & 1 
\end{smallmatrix} \right)\left( \begin{smallmatrix}
 \eta_{\p\mathfrak{b}^{2}}^{-\ell}a_{s} & \eta_{\p\mathfrak{b}^{2}}^{-\ell} \widetilde {b_{s}}\\
0 & a_{s}^{-1}
\end{smallmatrix} \right)
$
%\end{align*} 
then \begin{align*}
 &F_{\ai,r}f_{\mathfrak a,s}(g_{\infty}) = \frac{N(\p)^{s} }{v(\ai)} \int_{\R^{d}/\ai\p^{\ell - 2s}}\!\!\!\!\!\!\!\!\!\!\!\!e^{-2\pi i S(rx)}f_{\ai}\left(n(x\eta_{\p \mathfrak{b}^{2}}^{-\ell}a_{s}^{2}) \left(\begin{matrix}
\eta_{\p \mathfrak{b}^{2}}^{-\ell}a_{s} & \eta_{\p \mathfrak{b}^{2}}^{-\ell}\widetilde{b_{s}}\\
0 & a_{s}^{-1} 
\end{matrix} \right) g_{\infty} \right)dx. \\
\end{align*}
We now make the change of variables   $x \mapsto x\eta_{\p \mathfrak{b}^{2}}^{\ell}a_{s}^{-2}$.
 Since $\eta_{\p \mathfrak{b}^{2}} \in \mathfrak p\mathfrak{b}^{2}$ and $a_{s}\in \p^s \mathfrak{b}^{\ell}$, then 
 %and use that  
 $\eta_{\p \mathfrak{b}^{2}}^{-\ell}a_{s}^{2}\ai \p^{\ell - 2s} \subseteq \ai$, hence
% \mathcal{O}_{F}$, 
% since  $ \eta_{\p\mathfrak{b}^{2}}^{\ell}\mathcal O_F = \p^{\ell}\mathfrak{b}^{2\ell}$ %and $a_{s} \in \p^{s}\mathfrak{b}^{\ell}$.
\begin{align*}
 F_{\ai,r}f_{\mathfrak a,s}(g_{\infty}) 
&= \frac{N(\p)^{s} }{v(\ai)}\frac{N(\eta_{\p \mathfrak{b}^{2}})^{\ell}}{N(a_{s}^{2})} \int_{\R^{d}/\ai}\!\!\!\!e^{-2\pi i S\left(\frac{rx\eta_{\p \mathfrak{b}^{2}}^{\ell}}{a_{s}^{2}}\right)}f_{\ai}\left(n(x) \left(\begin{matrix}
\eta_{\p \mathfrak{b}^{2}}^{-\ell}a_{s} & \eta_{\p \mathfrak{b}^{2}}^{-\ell}b_{s}\\
0 & a_{s}^{-1} 
\end{matrix} \right) g_{\infty} \right)dx \\
& = \frac{N(\p)^{s} N(\eta_{\p \mathfrak{b}^{2}})^{\ell}}{N(a_{s}^{2})}F_{\ai, {r\eta_{\p\mathfrak{b}^{2}}^{\ell}}/{a_{s}^{2}}}\,f_{\ai}\left(\left( \begin{matrix}
      a_{s}\eta_{\p\mathfrak{b}^{2}}^{-\ell} & \widetilde{b_{s}}\eta_{\p\mathfrak{b}^{2}}^{-\ell}\\
      0 & a_{s}^{-1}
      \end{matrix}\right)g_{\infty}\right).
\end{align*}
\end{proof}

We thus obtain, using the previous proposition and \eqref{relacion de los tpl}: \begin{corollary} \label{coro:fourierterm tpl} Let  $\p$ be a  prime ideal, $\p \nmid \mathfrak{I}$, as in Proposition \ref{Fourier terms}, $\mathfrak b$ an  integral ideal  such that $\p\mathfrak{b}^{2}= \eta_{\mathfrak{pb}^{2}} \mathcal O_F$ with $\eta_{\mathfrak{pb}^{2}}$  totally positive. Assume $r \in \p^{\ell}\ai^{-1}\mathfrak{d}^{-1}.$ Then, if $f\in  \FS_q$,  
\begin{align}\label{eq:FourierTpf}
F_{\ai,r}\left((T_{\p^{\ell}}f)_{\ai}\right)(g_\infty) =  \chi^{-1}(\pi_{\mathfrak{b}^{\ell}})  \sum_{s = 0}^{\ell} \frac{N(\p)^{s} N(\eta_{\p \mathfrak{b}^2})^{\ell}}{N(a_{s}^{2})} \chi_f (k_{s}) F_{\ai, {r \eta_{\p\mathfrak{b}^2}^{\ell}}/{a_{s}^{2}}}\,f_{\ai}\left(\left( \begin{matrix}
      a_{s}\eta_{\p\mathfrak{b}^2}^{-\ell} & \widetilde{b_{s}}\eta_{\p\mathfrak{b}^2}^{-\ell}\\
      0 & a_{s}^{-1}
      \end{matrix}\right)g_{\infty}\right).
\end{align}
\end{corollary}

\

Since all  $T_{\p^{\ell}}$ are bounded selfadjoint operators  commuting with  the  Casimir operators  $C_{j}$,  we may choose an  orthogonal system $\lbrace V_{\varpi} \rbrace$ of  irreducible  subspaces such that   $T_{\p^{\ell}}V_{\varpi} \subset V_{\varpi}$ for every  $\varpi$. Thus,  $T_{\p^{\ell}}$ acts by a scalar   $\widetilde \lambda_{\p^{\ell}}(\varpi)\in \R$ on $V_{\varpi}$.
% and $\lambda_{\varpi, \p^{\ell}} \in \R$ is a  Hecke eigenvalue.
 Let $\lambda_{\p^{\ell}}(\varpi) := \tfrac{\widetilde \lambda_{\p^{\ell}}(\varpi)}{N(\p)^{\ell/2}}$ denote the normalized Hecke eigenvalue, thus $\lambda_{\p^{\ell}}(\varpi) \in [-2,2]$. We  now relate the normalized  eigenvalues $\lambda_{\p^{\ell}}(\varpi)$ 
 %$\lambda_{\varpi, \p^{\ell}}$ 
 of  $T_{\p^{\ell}}$ with the  Fourier  coefficients of the eigenfunctions $f$ in $V_\varpi$.
%\footnote{\olive decir que para cada $\ai$ hay una $\varpi_{\ai}$ }
\begin{theorem} \label{teo: relacion coeficientes y autovalores} Let $\p$ be a  prime ideal as in Proposition \ref{Fourier terms},  i.e $\p \nmid \mathfrak{I}$, $\p\mathfrak{b}^{2}= \eta_{\mathfrak{pb}^{2}}\mathcal O_F$ with $\eta_{\mathfrak{pb}^{2}}$  totally positive. Let $r \in \p^{\ell}\ai^{-1}\mathfrak{d}^{-1}$. 
Then, for   $f$ in an   irreducible subspace $V_{\varpi}$ and  $\ai$ any fractional ideal, we have
\begin{align*}
\lambda_{\p^{\ell}}(f) c^{\ai, r}(f) =  \chi^{-1}(\pi_{\mathfrak{b}^{\ell}})   \frac{1}{N(\p)^{\ell/2}}\sum_{s = 0}^{ \ell} \frac {N(\p)^{s}}{|N(a_s)|}N(\eta_{\p\mathfrak{b}^{2}}^{\ell/2}) e^{2\pi i S\left(\frac{r\widetilde{b_{s}}}{a_{s}}\right)}\chi_f (k_{s})
%c^{{r\eta_\tfrac{\p\mathfrak{b}^{2}}^{\ell}{a_{s}^{2}}},\, \ai}(\varpi),
c^{\ai,{r\eta_{\p\mathfrak{b}^{2}}^{\ell}}/{a_{s}^{2}}}( f)
\end{align*}\
where $\widetilde{b_{s}} \in \mathcal{O}_F$ and	$a_s \in \p^s \mathfrak{b}^\ell$ satisfies $v_{\p}(a_s) = s + \ell v_{\p}(\mathfrak{b})$, for any $0 \le s \le \ell$.
	%Furthermore, $\eta_{\p\mathfrak{b}^2}$ is such that $\pi_{\eta_{\p\mathfrak{b}^2}} = \pi_{\p}\pi_{\mathfrak{b}}^2$.
%where %$\eta_{\p^{\ell}\mathfrak{b}^{2\ell}} \mathcal O_F = \p^{\ell}\mathfrak{b}^{2\ell}$, and 
%$a_s$
% \in \p^{s}\mathfrak{b}$
% and $\widetilde{b_{s}}$
 % \in F$.
%are as in Proposition~\ref{relacion de los tpl}. 
%\footnote{ SaquÃ© $\eta_{\p^{\ell}\mathfrak{b}^{2\ell}} \mathcal O_F = \p^{\ell}\mathfrak{b}^{2\ell}$, no deberÃ­a hacer falta usando $\eta_{\p\mathfrak{b}^{2}}^{\ell}$  .} 
\end{theorem}
\begin{proof}
%\footnote{citar antes}
Let $q \in \Z^d$. %be a weight of   $V_{\varpi}$. 
We use  \eqref{fourierterm} in the expression in Corollary \ref{coro:fourierterm tpl}. Then
\begin{align}\notag
\widetilde \lambda_{\p^{\ell}}(f)c^{\ai,r}&({f})\,d^{\ai,r}(q,\nu_{\varpi}) W_{q}(r, \nu_{\varpi}; g) = \\\label{eq:carvarpi}
& \chi^{-1}(\pi_{\mathfrak{b}^{\ell}})  \sum_{s = 0}^{\ell} \frac{N(\p^{s})N(\eta_{\p\mathfrak{b}^{2}}^{\ell})}{N(a_{s}^{2})}\chi_f (k_{s}) c^{\ai, {r\eta_{\p\mathfrak{b}^{2}}^{\ell}}/{a_{s}^{2}}}( f) d^{\ai,{r\eta_{\p\mathfrak{b}^{2}}^{\ell}}/{a_{s}^{2}}}(q, \nu_{\varpi})W_{q}\Big({r\eta_{\p\mathfrak{b}^{2}}^{\ell}}/{a_{s}^{2}}, \nu_{\varpi}; p_{s}g\Big)
\end{align}
where $p_{s} = \left( \begin{matrix}
      a_{s}\eta_{\p\mathfrak{b}^{2}}^{-\ell} & \widetilde{b_{s}}\eta_{\p\mathfrak{b}^{2}}^{-\ell}\\
      0 & a_{s}^{-1}
      \end{matrix}\right). $ 
			Now, by   \eqref{whittaker}  %and \eqref{eq:d^r}
			 we have that 
%\textcolor{blue}{la definici\'on 2.17 es esta:
\begin{equation}\notag W_{q}(r, \nu; \left(\begin{smallmatrix}z_\infty&0\\0& z_\infty\end{smallmatrix}\right) n(x)a(y)k) = {\psi}_{\infty}(rx)\phi_{q}(k)  \epsilon_{\varpi}(\sign(ry)) \prod_{j=1}^{d}W_{q_{j}\sign(r_{j} y_j)/2, \nu_{j}} (4\pi |r_{j}y_{j}|),
\end{equation}
%Como estamos en $\PGL_2$, la matriz diagonal puede escribirse 
%Now $\left(\begin{smallmatrix}a_s \eta_{\p^\ell,\mathfrak{b}}^{-1} & 0\\
%0 & a_s^{-1}\end{smallmatrix}\right) =\left(\begin{smallmatrix}a_s ^{-1} & 0\\
%0 & a_s^{-1}\end{smallmatrix}\right) \left(\begin{smallmatrix}a_s^2 \eta_{\p^\ell,\mathfrak{b}}^{-1} & 0\\
%0 & 1\end{smallmatrix}\right). $ 
Thus \begin{align*}
W_{q}\Big({r\eta_{\p\mathfrak{b}^{2}}^{\ell}}/{a_{s}^{2}}, \nu_{\varpi}; p_{s}g\Big) &=W_{q}\Big({r\eta_{\p\mathfrak{b}^{2}}^{\ell}}/{a_{s}^{2}}, \nu_{\varpi}; \left(\begin{smallmatrix}1 & {a_s \widetilde{b_s}}/{\eta_{\p \mathfrak{b}^{2}}^\ell}\\
0 & 1\end{smallmatrix}\right)\left(\begin{smallmatrix}a_s \eta_{\p \mathfrak{b}^{2}}^{-\ell} & 0\\
0 & a_s^{-1}\end{smallmatrix}\right)g\Big)\\
&=  W_{q}\Big({r\eta_{\p\mathfrak{b}^{2}}^{\ell}}/{a_{s}^{2}}, \nu_{\varpi}; \left(\begin{smallmatrix}1 &\, {a_s \widetilde{b_s}}/{\eta_{\p \mathfrak{b}^{2}}^\ell}\\
0 & 1\end{smallmatrix}\right)\left(\begin{smallmatrix}a_s^2 \eta_{\p \mathfrak{b}^{2}}^{-\ell} & 0\\
0 & 1\end{smallmatrix}\right)g\Big)\\
&= {\psi}_{\infty}\left(\tfrac{r\eta_{\p \mathfrak{b}^{2}}^\ell}{a_s^2}\tfrac{a_s\widetilde{b_s}}{\eta_{\p\mathfrak{b}^{2}}^\ell}\right)\epsilon_{\varpi}(\sign(\tfrac{r_j\eta_{\mathfrak{pb}^2}^{\ell}y_j}{a_s^2})) \prod_{j=1}^{d}W_{q_{j}\sign\left (\tfrac{r_{j}{\eta_{\p \mathfrak{b}^{2}}^\ell} y_j}{{a_s^2}}\right)/2, \nu_{j}} \!\!\!\!\!\! (4\pi |\tfrac{r_{j}{\eta_{\p \mathfrak{b}^{2}}^\ell}y_j}{{a_s^2}}\tfrac{ a_s^2} {\eta_{\p\mathfrak{b}^{2}}^{\ell}}|) \\
&= e^{2\pi i S\left(\frac{r\widetilde{b_{s}}}{ a_{s}}\right)}  \epsilon_{\varpi}(\sign(\tfrac{r_j\eta_{\mathfrak{pb}^2}^{\ell}y_j}{a_s^2})) \prod_{j=1}^{d}W_{q_{j}\sign\left (r_{j}{\eta_{\p\mathfrak{b}^{2}}^\ell} y_j/{a_s^2}\right)/2, \nu_{j}} (4\pi |r_{j}y_j|)
\end{align*}
 and since $\eta_{\p \mathfrak{b}^{2}}^\ell$ is totally positive, then $\sign\left(r_{j}{\eta_{\p\mathfrak{b}^{2}}^\ell} y_j /{a_s^2}\right)= \sign{(r_j y_j)}$. 

Hence we have the identity
\begin{align}\label{eq:Wreta}
& W_{q}\Big({r\eta_{\p\mathfrak{b}^{2}}^{\ell}}/{a_{s}^{2}}, \nu_{\varpi}; p_{s}g\Big) = e^{2\pi i S\left(\frac{r\widetilde{b_{s}}}{ a_{s}}\right)} W_{q}(r, \nu_{\varpi}; g) \;.
\end{align}
%: \begin{equation}\notag d^{r}(q,\nu):= \frac{1}{\sqrt{2^{d}|D_{F}N(r)|}} \prod_{j=1}^{d} \frac{e^{\pi i q_{j}}}{\Gamma \left( \frac{1}{2}+\nu_{j}+ \frac{q_{j}}{2}\sign (r_{j}) \right)} % 
%\end{equation}  applied to $d^{\tfrac{r\eta_{\p^\ell,\mathfrak{b}}}{a_s^2}}
%(q,\nu)$.
On the other hand,  by \eqref{eq:d^r} we have 
\begin{align}\label{eq:drnu}
%& W_{q}\Big(\frac{r\eta_{\p\mathfrak{b}^{2}}^{\ell}}{a_{s}^{2}}, \nu_{\varpi}; p_{s}g\Big) = e^{2\pi i S\left(\frac{r\widetilde{b_{s}}}{ a_{s}}\right)} W_{q}(r, \nu_{\varpi}; g) .\\
d^{\ai,{r\eta_{\p\mathfrak{b}^{2}}^{\ell}}/{a_{s}^{2}}}(q, \nu_{\varpi}) = \frac{|N(a_{s})|}{N(\eta_{\p\mathfrak{b}^{2}}^{\ell})^{1/2}} d^{\ai,r}(q, \nu_{\varpi}).
\end{align}
%\footnote{Creo que lo anterior habrÃ­a que citarlo o justificarlo. Fijate que agreguÃ© un cuadrado en a, va o no?} 
Substituting \eqref{eq:Wreta} and \eqref{eq:drnu} in \eqref{eq:carvarpi},  the assertion in the theorem follows. 
\end{proof}

\section{Asymptotic distribution of Casimir and Hecke eigenvalues }

\subsection{Plancherel measures}

In this subsection we  recall some facts  on the Plancherel measure.  We have   $\textrm{Pl} = \otimes_j \textrm{Pl}_{\xi_j}$  on $\R^d$, where $\textrm{Pl}_{\xi_j}$ are  the measures on $\R$ given by
\begin{align*}
&\textrm{Pl}_0(f) = \int_{1/4}^{\infty} f(\lambda) \tanh \pi \sqrt{\lambda - \tfrac 1 4} d\lambda + \sum_{b\geq 2, b\equiv0 \textrm{ mod }2} (b-1) f(\tfrac b 2(1-\tfrac b 2)), \\
& \textrm{Pl}_1(f) = \int_{1/4}^{\infty} f(\lambda) \coth \pi \sqrt{\lambda - \tfrac 1 4} d\lambda + \sum_{b\geq 3, b\equiv 1 \textrm{ mod }2} (b-1) f(\tfrac b 2(1-\tfrac b 2)). 
\end{align*}
In particular, $\textrm{Pl}_{\xi_j}$ gives zero measure to the set of \textit{exceptional eigenvalues} in $[0, \tfrac 1 4)$. 

We also need another measure, denoted  $V_{1}$ as in \cite{BM10}, with a  product structure $V_{1}=\otimes_{j} V_{1, \xi_{j}}$ where
%\footnote{{\red otra vez lo de los $\xi_j$}}   with 
\begin{align*}
&\int hdV_{1,0} = \tfrac 1 2 \int_{5/4}^{\infty} h(\lambda)d\lambda + \tfrac 1 2 \int_{0}^{5/4}|\lambda - \tfrac 1 4|^{-1/2} d\lambda + \sum_{\beta \in \N + \tfrac 12} \beta h(1/4 - \beta^2),\\
& \int hdV_{1,1} = \tfrac 1 2 \int_{5/4}^{\infty} h(\lambda)d\lambda + \tfrac 1 2 \int_{1/4}^{5/4}|\lambda - \tfrac 1 4|^{-1/2} d\lambda + \sum_{\beta \in \N} \beta h(1/4 - \beta^2).
\end{align*}
\noindent The measure $V_1$  is comparable to the Plancherel measure $\Pl$ for sets  that  %points $\lambda$ such that all  coordinates of its points 
are at a positive  distance off $(0, 1/4)$ (see \cite[\S1.2.2]{BM10}).  %The error term  $o(V_1 (\Omega_t))$ is small in comparison to $\Pl (\Omega_t)$.

Often it is convenient to use, instead of $\lambda(f) \in \R^d$, the corresponding spectral parameter $\nu(f) \in ([0,\infty) \cup i(0, \infty))^d$, $\nu(f) = \sqrt {\tfrac 14 -\lambda(f)}$. We will use a tilde to indicate that the relevant measures like $\widetilde N ^r$, $\widetilde \Pl$ and $\widetilde V_1$ are taken with respect to the variable $\nu$, and we will write
%Let us denote the  function 
\begin{equation}
\widetilde{N}^r (\widetilde{\Omega}) = \sum_{f\in \mathcal B_{\chi,q}\atop \nu(f) \in \widetilde{\Omega}} |c^{\ai,r} (f)|^{2},
\end{equation}
for sets $\widetilde{\Omega} \subset Y_{\xi}= \prod_{j=1}^d Y_{\xi_j}$, where $\xi_j \in \{0,1\}$,
%\footnote{Check this, i.e. the $\xi_j $.} 
\begin{align}\notag
Y_0 = \lbrace \tfrac{b-1}2 : b\ge 2, \textrm { b even}\rbrace \;\textrm{ and }\;   %\cup i[0, \infty) \cup (0,\nu_0]\\
Y_1 = \lbrace \tfrac{b-1}2 : b\ge 2, \textrm { b odd}\rbrace. %\cup i[0, \infty). 
\end{align}
%and $\nu_0 = \sqrt {\tfrac 14 -\lambda_0}$.

In the $\nu$-coordinate the Plancherel measure on $Y_{\xi}$ is now given by $\widetilde{\textrm{Pl}}_{\xi} = \otimes_j \widetilde{\textrm{Pl}}_{\xi_j}$, where 
\begin{align}\label{tildePl}
&\widetilde{\textrm{Pl}}_{0}(f) = % i\int_{\textrm{Re}\nu = 0} f(\nu)\nu \tan \pi \nu d\nu+ 
\sum_{\beta \in \tfrac 1 2 + \Z} |\beta|f(\beta),\quad 
\widetilde{\textrm{Pl}}_{1}(f) =
% -i\int_{\textrm{Re}\nu = 0} f(\nu)\nu \cot \pi \nu d\nu + 
\sum_{\beta \in \Z\setminus \lbrace 0 \rbrace} |\beta|f(\beta).
\end{align}
The measure $\widetilde{V_1}$ has a product form $\widetilde{V_1}= \otimes_j \widetilde{V_{1,\xi_j}}$ on the space $\widetilde{\Omega} \subset Y_{\xi}= \prod_{j=1}^d Y_{\xi_j}$, where $\xi_j \in \{0,1\}$ with
%$((0,\infty) \cup i[0,\infty))^d$ with
\begin{align}\label{measuretildeV1}
\int h\, d\widetilde{V_{1,0}} =
% \int_1^{\infty} t h(it) dt + \int_0^1 h(it)dt + \int_0^{\nu_0}h(x)dx + 
\sum_{\beta>0,\,\, \beta \equiv \tfrac 1 2 (1)} \beta h(\beta) \;\textrm{ and }\;
\int h\, d\widetilde{V_{1,1}} =
%\int_1^{\infty} t h(it) dt + \int_0^1 h(it)dt + 
\sum_{\beta>0, \,\, \beta \equiv 0 (1)} \beta h(\beta).
\end{align}

\subsubsection{Test functions.} \label{testfunctions}
We {\ shall } use test functions  %is the same as in \cite[\S2.1.1]{BM10}. They 
of product type $\varphi(\nu) = \prod_{j} \varphi_j(\nu_j)$, where the factor $\varphi_j$ is defined on a strip $|\textrm{Re}\nu_j|\leq \tau$ with $\tfrac{1}{4} < \tau < \tfrac{1}{2}$, and also on the discrete set $\tfrac{1+\xi_j}{2} + \N_0$. The $j^{th}$-factor $\varphi_j$ satisfies  $\varphi_j(\nu_j) \ll (1+|\nu_j|)^{-a}$ on its domain for some $a>2$, and is even and holomorphic on the strip $|\textrm{Re} \nu_j| \leq \tau$. The $\nu_j$ occurring are related to spectral data. The eigenvalues $\lambda_{j}(\varpi)$ of the Casimir operators in $V_{\varpi}$ are of the form $\lambda_j(\varpi) = \tfrac{1}{4} - \nu_j(\varpi)^2$,   with $\nu_j(\varpi) \in (0,\infty) \cup i[0,\infty)$. Thus, the test functions $\varphi$ can be viewed as defined on a neighborhood of the set of possible values of the vectors $\nu(\varpi) = (\nu_j(\varpi))_j$. %We thus have (see \cite{Ve04} or \cite{Ma13})

\subsection{The asymptotic formula}

 In this subsection we will first review the Kuznetsov sum formula for $\GL_2$
	%\footnote{\olive en lugar de $\PGL_2$}. 
	and then use it 
	to derive the asymptotic formula \eqref{eq:asymptotic}, which 
	will be a central tool in the proof of the main theorems in this paper.

A sum formula for $\SL_2$ over a number field for $K$-spherical functions was given in \cite [Theorem 6.1.]{BM98}.
In \cite{BMP01} (see also \cite[Theorem 3.21]{BM09}) an extension valid  for arbitrary $K$-types was given, still in the case of  $\SL_2$.
A spherical version  for  $\GL_2(\A_F)$ was proved  in \cite [Proposition 1]{Ve04} 
%In both these cases,  $K$-spherical functions were considered. 
and a $\GL_2$ general version was derived in \cite [Theorem 1]{Ma13} (see also \cite[\S 2.12]{BH10}).

 \begin{definition}\label{def: BFSchiq} Fix $q \in \Z^d$ and $\chi$ a character of $\A_F^\times/F^\times$.	Let  $\mathcal B_{\chi,q}$ be  an orthonormal basis of the space 
  $\textrm {FS}_{\chi, q}^{\textrm {disc}}$ so that each $f\in \mathcal B_{\chi,q}$ has $K_\infty$-weight $q$ and is an eigenfunction  of all the Casimir operators $C_j$,  of all the Hecke operators $T_\mathfrak p$ and an eigenfunction of the center $Z(\A_F^\times)$ by the central character $\chi$.
\end{definition}

\ 

We define the  \emph{twisted Kloosterman sum} as in \cite[Definition 2]{Ve04}, taken in the particular case $\mathfrak a_1= \mathfrak a_2 = \mathfrak a$. 
%It differs from that of \cite{BM10} by necessity, since we must include ideal classes as parameters. 
%\footnote{Mencionar $\chi$. Revisar las constantes $C_1$ y $C_2$. Ver  \cite[\S 2.12]{BH10}}

Let $\mathfrak{a}$ be a fractional ideal of $F$ 
and let $\mathfrak{c}$ be any ideal so that $\mathfrak{c}^2 \sim \mathfrak{a}^2$  in the  class group. Fix $c \in \mathfrak{c}^{-1}\mathfrak{I}$, $r \in \mathfrak{a} ^{-1} \mathfrak{d}^{-1}$ and $r' \in \mathfrak{a} \mathfrak{d}^{-1} \mathfrak{c}^{-2}$. We set
\begin{equation}
KS(r,\mathfrak{a} ; r',\ai; c , \mathfrak{c}) = \sum_{x \in (\mathfrak{a} \mathfrak{c}^{-1} / \mathfrak{a} (c))^{\times}} e^{\frac{rx+r' x^{-1}}{c}} \overline{\chi(x)},
\end{equation} 
where the summation runs through  elements $x$  generating $(\mathfrak{a} \mathfrak{c}^{-1} / \mathfrak{a} (c))$ as an $\mathcal{O}_F$-module, %and where
 $x^{-1}$ is the unique element in $(\mathfrak{a}^{-1} \mathfrak{c} / \mathfrak{a}^{-1} (c)\mathfrak{c}^{2})^{\times}$ such that $xx^{-1} \in 1 + c\mathfrak{c}$  and $\chi$ is a character of $\A_F^{\times}/F^{\times}$ with conductor dividing $(c)\mathfrak{c}$ (then $\chi$ induces in a natural way a function $(\mathfrak{a} \mathfrak{c}^{-1} / \mathfrak{a} (c))^{\times} \rightarrow \C$).
 
Twisted Kloosterman sums satisfy a Weil bound (see \cite[(13)]{Ve04}). Namely
\begin{equation}\label{eq: weil bound}
|KS(r,\mathfrak{a};r',\ai;c,\mathfrak{c})| \ll N(r\mathfrak{a}\mathfrak{d},r'\mathfrak{c}^2 \mathfrak{a}^{-1}\mathfrak{d},c\mathfrak{c})^{1/2} N(\mathfrak{c}c)^{1/2 + \varepsilon}
\end{equation}
where the brackets $(\cdot,\cdot,\cdot)$ denote greatest common divisors of ideals. 

We now give the version of the Kuznetsov formula for  $\PGL_2$ 
that we shall use (see \cite[Prop.2.1]{Ve04} in the spherical case and   \cite[Thm. 1]{Ma13}  for arbitrary $K_\infty$-types).
 
%\begin{theorem}
Let $\mathfrak{a}$ be	 a fractional ideal, and $r, r'\in \mathfrak{a}^{-1}\mathfrak{d}^{-1}$.

	%Let $\mathfrak{a}_1,\mathfrak{a}_2$ be fractional ideals, and $r \in \mathfrak{a}_1^{-1}\mathfrak{d}^{-1}$, $r' \in \mathfrak{a}_2^{-1}\mathfrak{d}^{-1}$. 
	Then, for any test function $\varphi$ as in \S\ref{testfunctions}, we have  
	\begin{align}\label{eq:sum formula}
	&\sum_{f \in  \mathcal \mathcal B_{\chi,q}} c^{\mathfrak{a},r}(f)\overline{c^{ \mathfrak{a},r'}(f)} \varphi(\nu(f)) + \textrm{CSC} = \frac{ 2^d\sqrt {D_F}}{\pi^d h_F} \tilde \delta_{r,r'}
	 \widetilde {\Pl} (\varphi)\\& \notag  + \frac{ 2^{d-1}}{h_F} \sum_{\mathfrak{c} : \mathfrak{c}^2 \sim \mathfrak{a}^2} \sum_{\varepsilon \in \mathcal{O}^* / (\mathcal{O}^*)^2} \sum_{c \in \mathfrak{c}^{-1}\mathfrak{I}} B\varphi(\nu(f), \tfrac{\varepsilon rr'}{c^2 \gamma}) \frac{KS (r, \mathfrak{a} ; \varepsilon r' \langle \mathfrak{c}^2 /\mathfrak{a}^2 \rangle^{-1}, \ai;c,\mathfrak c)}{N(c\mathfrak{c})}
	\end{align}% 
	where
	$\tilde \delta_{r, r'} =1$, iff   $r^{-1} r'\in \mathcal O_F^{*+}$ 
	%in $\mathcal{C}^{+}_F$ 
	and is equal to $0$ otherwise. Also, CSC denotes a contribution of the continuous spectrum and $B\varphi$ is a Bessel transform of $\varphi$. % (see \cite[\S2.1.4]{BM10}).
%\end{theorem}
	
We now describe the main procedure in the derivation (see \cite[Section 4.2]{Ma13}).  
%As explained in \cite[Remark 3]{Ve04} if $\mathfrak a_1 \mathfrak a_2$ is not a square in $C_F$ then the sum is zero. 
%We thus  fix two nonzero ideals $\mathfrak{a}, \mathfrak{b}$ and let $\ai_1= \ai$ and $\ai_2 =\ai \mathfrak b^2$ and consider the following characters of $N(\R)^d$. 
For  $r, r' \in \ai^{-1}\mathfrak{d}^{-1}$, 
%$r' \in \ai^{-1}\mathfrak{b}^{-2}\mathfrak{d}^{-1}$  
nonzero elements with the property that $r / r'$ is totally positive, let $\psi_\infty(rx)$ and $\psi_\infty(r'x)$,
%$\chi_r$ and $\chi_{r'}$, 
 be characters on $N(\R)^d$ 
 %given as in \eqref{character-r}, 
which are trivial on $\Gamma(\mathfrak{I},\mathfrak{a})_N$. 
%and   $\Gamma_N( \mathfrak{I},\mathfrak{ab}^2),$ 
% respectively. 

Define functions $f_1, f_2$ on $Z_{\infty} \ba \GL_2 (\R)^d$ such that
if $g\in \GL_2(\R)^d$, $x \in \R^d$ and $z \in {\R^\times}^d$ then
\begin{align*}
&f_1\left(\left(\begin{matrix}
1 & x\\ 0 & 1
\end{matrix} \right)\left(\begin{matrix}
z & 0\\ 0 & z
\end{matrix} \right)g\right) = \psi_{\infty}(rx) f(g),\quad
f_2\left(\left(\begin{matrix}
1 & x\\ 0 & 1
\end{matrix} \right)\left(\begin{matrix}
z & 0\\ 0 & z
\end{matrix} \right)g\right) = \psi_{\infty}(r'x) f(g).
\end{align*}
One now defines Poincar\'e series $P_1$ and $P_2$ on the $h_F$ classical components in the right-hand side of  \eqref{eq:L2classical} in such a way that $P_1^{\mathfrak{a'}}=0$ for $\mathfrak{a'}\ne \mathfrak{a}$, $P_2^{\mathfrak{a'}}=0$ for $\mathfrak{a'}\ne \mathfrak{ab^2}$ and, furthermore, 
\begin{align*}
P_1^{\mathfrak{a}}(g) = \sum_{\gamma \in Z_{\Gamma}\Gamma_{N}( \mathfrak{I},\mathfrak{a})\ba \Gamma_0( \mathfrak{I}, \mathfrak{a})} f_1(\gamma g),\quad
 P_2^{\mathfrak{ab}^2}(g) = \sum_{\gamma \in Z_{\Gamma}\Gamma_{N}( \mathfrak{I},\mathfrak{ab}^2)\ba \Gamma_0( \mathfrak{I},\mathfrak{ab}^2)} f_2(\gamma g),
\end{align*}
with $Z_\Gamma = Z_\infty \cap \Gamma$ and $\Gamma_N = \Gamma \cap N(\R)^d$.

The sum formula emerges by computing the inner product $\langle \pi_{\mathfrak{b}}P_2, P_1 \rangle$ in two ways, geometrically and spectrally  and by averaging over $\mathfrak{b} \in \mathcal{C}_F$. We refer to \cite[sections 6.3, 6.4]{Ve04}  and \cite [sections 5.1, 5.2]{Ma13} for the geometric description and for calculations on the spectral side, and to \cite{BM98} or \cite{BM09} for convergence considerations and  for the integral formulas used. See also \cite[\S 2.12]{BH10}.

\

%\subsection{Asymptotic Formula} 
%\footnote{Ver si se pueden tomar funciones de prueba Z-invariantes. PMaga las usa.}
As an application of formula \eqref{eq:sum formula} we  will derive an asymptotic formula that is an adaptation of a result in \cite{BMP03} and \cite{BM10}, and which  will be a main
 tool in the proof of Theorem~\ref{thm:principal 1}. 
We fix a   partition of the  set of  archimedean places of $F$
\begin{equation}
\lbrace 1, \ldots, d \rbrace = E \sqcup Q_+ \sqcup Q_- ,
\;\; \textrm{ where } Q := Q_+ \sqcup Q_- \neq \emptyset,
\end{equation}
and  we consider  the automorphic irreducible subrepresentations  $\varpi$ in $L^{2,\textrm {disc}} \big(\Gamma_{0}(\mathfrak{I}, \ai) \ba \PGL_2 (\R)^d, \chi \big)$ with prescribed   spectral  parameters at the places  $j \in E$, so that, for  $j \in Q_-$, $\varpi_j$ is in the discrete series and  for  $ j \in Q_+$, $\varpi_j$ is in the   principal or complementary series. That is, we consider the set $\mathcal{R}$ of representations $\varpi$ such that  
\begin{align}
&\lambda_{j}(\varpi) \in [a_j , b_j] \textrm{ for } j\in E,   \quad  \lambda_{j}(\varpi) > 0  \textrm{ if } j \in Q_+, \; \textrm{ and }
\lambda_{j}(\varpi) \leq 0   \textrm{ if } j \in Q_-.
\end{align} 
Furthermore, the intervals $ [a_j , b_j] \subset \R$ are required to satisfy the  condition that the endpoints  $a_j, b_j$ are not of the form  $\tfrac{b}{2}\left(1 - \tfrac{b}{2}\right)$ with $b \geq 1$,  $b \equiv \xi_{j}$ mod $2$. %\footnote{{\red esto de los $\xi_j$ es para $\SL_2$ no para $\GL_2$}}

Let \begin{equation}\label{eq: Omega general}
\Omega_t = \prod_{ j \in E} [a_j, b_j] \times \prod_{ j \in Q_+} [A_j(t), B_j(t)] \times \prod_{ j \in Q_-} [C_j(t), D_j(t)],
\end{equation}
where  $\Omega_t$ satisfies the conditions in Proposition~6.2 of \cite{BM10}, in particular $B_j(t) \rightarrow +\infty$ for at least one $j \in Q_+$ or $C_j(t) \rightarrow -\infty$ for some $j \in Q_-$.

%We now state an asymptotic result that is an extension of   \cite[Thm~ 1.3]{BM10} and will be a main tool in the proof of Theorem~\ref{thm:principal 1}.
%{\red sacar\'ia este p\'arrafo}
%One may also  express $V_{1}$ in the variable $\nu$, where $\lambda = \tfrac 14 - \nu^2$ with $\nu \in (0,\infty) \cup i[0,\infty)$. If we denote $\widetilde{V}_{1} = \otimes_j \widetilde{V}_{1, \xi_{j}}$ the corresponding measure on   
%$\big( (0, \infty) \cup i[0, \infty) \big)^d$,  then $\widetilde{V}_{1}$ is positive on $Y_{\xi}$ and $\widetilde{Pl}(\widetilde{\Omega}) \ll \widetilde{V}_{1}(\widetilde{\Omega})$ for every  $\widetilde{\Omega}$. 

\begin{theorem}[Asymptotic formula] Let $t \mapsto \Omega_t$ be a family of bounded sets in $\R^d$ as in \eqref{eq: Omega general}.
	Let  $\mathfrak{a}$ be a fractional ideal, $\chi$ a unitary character of $\A^\times_F/F^\times$ and $q \in \Z^d$.
	Let $\mathcal B_{\chi,q}$ be as in Definition \ref{def: BFSchiq}. 
	%an orthonormal basis  of the discrete spectrum of $FS_{\chi, q}$ such that $Z(A_F)$ acts on each $f$ by an  eigenvalue $\chi$. 
	Then, if $r , r' \in \mathfrak{a}^{-1}\mathfrak{d}^{-1}$, as $t \longrightarrow \infty$
	\begin{align} \label{eq:asymptotic}\sum_{f \in \mathcal B_{\chi,q} \atop \lambda(f)  \in \Omega_t} \overline{c^{\mathfrak{a}, r}(f)}c^{\mathfrak a ,r'}(f) =  \tilde\delta_{ r, r'}\frac{ 2^d\sqrt {D_F}}{\pi^d h_F}  \Pl (\Omega_t) + o(V_1(\Omega_t)),	\end{align}
	where
	$\tilde \delta_{r, r'} =1$, if   $r/r' \in (\mathcal{O}_F^\times)^{+}$ and $\tilde \delta_{r, r'} =0$ otherwise.
	\label{thm: asymptotic formula}
\end{theorem} %\footnote{{\red Lo mismo que el teorema anterior sobre la sumatoria y se repite en el teorema 6.3}}

%In \cite[Thm 1.3]{BM10} these type of asymptotic formulas  were obtained for  the counting  function 
%\begin{equation}\label{eq:Nr}
%N^{r}(\Omega_t) := \sum_{f\in B(FS_{\chi,q}) \atop \lambda(C_j,f)\in \Omega_t} |c^{r}(f)|^2, \textrm{ with } r \in \mathfrak{d}^{-1} \ba \lbrace 0 \rbrace,
%\end{equation}
%for many different choices other than the sets $\Omega_t$ in \eqref{eq: Omega general}.
%The  representations $f$ run over the orthogonal system of irreducible subspaces  of $L^{2, \textrm{cusp}}_{\xi}\big(\Gamma_0(\mathfrak{I})\ba \SL_2(\R)^d, \chi_{_{f}}\big)$,  with  Casimir $C_j$ eigenvalues in the region $\Omega_t$  of the eigenvalue space.

The  asymptotic formula is derived by using \eqref{eq:sum formula} (with the variable $\lambda$ in place of $\nu$) for a suitable choice of the test functions.
We can adapt the method of proof of \cite[Thm 1.3]{BM10}. 
 A main part  of the argument involves showing that the CSC and  the Kloosterman term of the Kuznetsov formula are of lower order than the delta term, which is the main term, given  in \eqref{eq:sum formula}. 
 For simplicity we defer a sketch of these computations to the  Appendix
 %where we show that each one of these terms is $o(V_1(\Omega_t))$  
 (see \eqref{eq:bound Kloosterman} and \eqref{eq:bound Eisenstein}).  
%We quote, as in \cite{BM10}, \cite{BMP03} and \cite{BM09} for the estimates of the CSC. 
%For the CSC,  we have the following bound $$D^{r} (\lambda, \chi; iy, i\mu) \ll_{F,\mathfrak{I}, r, r'} (\textrm{log}(2 + \sum_{j}|y+\mu_j|) )^7 $$
%of the Fourier coefficients of Eisenstein series (see \cite[(33)]{BM10}).
%To prove the  bound of the Kloosterman term, we use the Weil bound \eqref{eq: weil bound} for the Kloosterman sums. The extra finite sums in the Kloosterman term do not affect the order of the estimate. 

 In the proof, we choose the test functions in the same way as in \cite[Lemma 2.2]{BM10}. 
 Then the asymptotic result follows by arguing as in
 %by approximating the characteristic function of $\Omega_t$ by test functions as in  \cite[Theorem 3.3]{BMP03}  
\cite[sections \S 2-5]{BM10}.
%\footnote{Esto hay que expandirlo y explicarlo bien y ver los error terms. Por ahora estÃ¡ dicho vagamente.} 
\begin{remark} We note that the formula does not  exclude 
exceptional eigenvalues but shows that the exceptional spectrum has 
 lower density, since the interval $(0,1/4)$ is not in the support of the Plancherel measure appearing in the main term in \eqref{eq:asymptotic}.
\end{remark}

\subsection{Joint distribution of Hecke and Casimir eigenvalues} \label{sec: distribucion conjunta}
The main goal of this section will be to estimate the function $N^r (\Omega_t)$,  now including  conditions on the Hecke eigenvalues $\lambda_{\p}(f)$.

Given a subinterval $I_{\p} \subseteq [-2, 2]$ and an integral ideal $\ai$ we study the asymptotic behavior of the function 
\begin{equation}\label{eq: N(omega, I)}
N^{r}(\Omega_t, I_{\p}) = \sum_{ f \in \mathcal B_{\chi,q}: \lambda(f) \in \Omega_t, \atop  \lambda_{\p}(f) \in I_{\p}} |c^{\ai,r} (f)|^2
\end{equation}
with $r \in \ai^{-1}\mathfrak{d}^{-1} \setminus \lbrace 0 \rbrace$ and  $f$ running through an orthogonal  system  of  irreducible subspaces of $L^{2, \textrm{cusp}}\big( \Gamma_{0}(\mathfrak{I}, \ai) \ba \PGL_2(\R)^d, \chi \big)$.

Relative to the Hecke eigenvalues, one wishes to measure their  distribution relative to the   Sato--Tate measure \eqref{eq:SatoTate}. As in  \cite{KL08} and \cite{Li09} we will use the following measure for any $r\in F$
\begin{equation}\label{eq: phi}
\Phi_{\ai,r}(x) := \sum_{\ell'=0}^{\textrm{ord}_{\p}(r\ai \mathfrak{d})} X_{2\ell'}(x)d\mu_{\infty}(x),
\end{equation}
where $d\mu_{\infty}$ is the   Sato--Tate measure in $[-2,2]$ %\eqref{eq:SatoTate} 
and $X_m(x)$ denotes the $m^{th}$-Chebyshev polynomial. These polynomials are orthonormal with respect to the Sato-Tate measure in $[-2,2]$ and satisfy  $X_\ell(T_{p}) =T_{\p^\ell}$ for every prime ideal $\p$ (see \cite[\S2]{Se97}). 
 %\footnote{{\red nos queda la constante $N(\p)^{\ell}$ porque no hemos normalizado los autovalores de Hecke}}
In particular, for any $r\in \ai^{-1}\mathfrak{d}^{-1}$ such that 
$\p\nmid r\ai\mathfrak{d}$, 
 %$\textrm{ord}_{\p}(r\ai\mathfrak{d})= 0$ %(in particular, a.e in $r$), 
$\Phi_{\ai,r}(x)$ is just the  Sato--Tate measure.
The following theorems are the main results in this paper. 
\begin{theorem} \label{thm:principal 1} Let $t \mapsto \Omega_{t}$ be a family of subsets of $\R^{d}$ as in (\ref{eq: Omega general})
	 %{\blue with $\Pl(\Omega_{t}) \rightarrow \infty$ as $t \rightarrow \infty$ 
	 	and let $\mathcal B_{\chi,q}$ be as in Definition \ref{def: BFSchiq}.  Let $\p \nmid \mathfrak{I}$, $\p$  a square in the narrow class group,  $\ai$ an integral ideal and, for each $f\in \mathcal B_{\chi,q}$, let $\lambda_{\p}(f)$ be the eigenvalue of   $T_{\p}$ in $f$.  Assume $ r\in \p^\ell \ai^{-1} \mathfrak d^{-1}$.
	%and $\Phi_r= \sum_{\ell'=0}^{\textrm{ord}_{\p}(r\mathfrak{d})} X_{2\ell'}(x)d\mu_{\infty}(x)$. 
	%If  $X_{\ell}$ is the  $\ell$--th  Chebyshev polynomial, 
	Then, as $t \rightarrow \infty$ we have 
	\begin{align}\label{eq:principal 1}
	\sum_{f\in \mathcal B_{\chi,q} \atop \lambda(f)\in \Omega_{t}} \!\!|c ^{\ai,r}(f)|^{2}X_{\ell}(\lambda_{\p}(f)) = \begin{cases} \frac{ 2^d\sqrt {D_F}}{\pi^d h_F} \,\Pl(\Omega_{t})\,\Phi_{\ai,r}(X_{\ell}) + o(V_{1}(\Omega_{t}))\,\,\, \textrm{ if  } \ell \textrm{ even,}  
	%0\leq \ell/2 \leq \textrm{ord}_{\p}(r\ai \mathfrak{d})
	\\
	o(V_{1}(\Omega_{t})) \,\, \textrm{ if  } \ell \textrm{  odd}, 
	\end{cases}  
	\end{align}%N(\p)^{\ell /2}
%\frac{2\sqrt{|D_{F}|}\Vol(\Gamma \ba G)}{(2\pi)^{d}}	
%	\footnote{{\red acÃ¡ ya puse constante del delta term del Venkatesh, asÃ­ uniformizamos, y agreguÃ© $N(\p)^{\ell/2}$}}
%	where ${\red C =   \frac{ 2^d\sqrt {D_F}N(\p)^{\ell /2}}{\pi^d h_F}}$ and 
	with $\Phi_{\ai,r}$  as in \eqref{eq: phi} and  $V_1$  as in Theorem~\ref{thm: asymptotic formula}.
%\footnote{Hay que normalizar los autovalores de Hecke para deshacernos de $N(\p)^{\ell/2}$}	

\end{theorem}
%{\blue Notar que no habÃ­amos escrito que cuando $\ell$ no es par el resultado era $o(V_{1}(\Omega_{t}))$}
\begin{proof}
	%By  Corollary~\ref{coro: chebyshev y autovalores} 
	
	%By a well known property of $X_\ell$, we have that    
	Since $ X_{\ell}(\lambda_{\p}(f)) = \lambda_{\mathfrak p^{\ell}}(f)$,
then,  by  Theorem~\ref{teo: relacion coeficientes y autovalores} and Theorem \ref{thm: asymptotic formula}
		% $$c^{\ai,r}(f) \lambda_{\p^{\ell}, f} =  \sum_{s = 0}^{\ell}N(\p)^{s} N(\eta_{\p^{\ell},\mathfrak{b}})^{1/2}e^{2\pi i S\left(\frac{rb_{s}}{a_{s}}\right)}c^{\ai,\frac{r\eta_{\p^{\ell}, \mathfrak{b}}}{a_{s}^{2}}}(f)\chi_{_{f}}(k_{\gamma_{s}}).$$
	\begin{align}\notag
		\sum_{f\in \mathcal B_{\chi,q} \atop \lambda(f)\in \Omega_{t}}\!\!\! |c^{\ai,r}(f)|^{2}X_{\ell}(\lambda_{\p}(f))& = \sum_{f\in \mathcal B_{\chi,q} \atop \lambda(f)\in \Omega_{t}} \overline{c^{\ai,r}(f)}c^{\ai,r}(f)\lambda_{\p^{\ell},f}\\
\notag	&= 	\sum_{f\in \mathcal B_{\chi,q} \atop \lambda(f)\in \Omega_{t}} %\chi^{-1}(\pi_{\mathfrak{b}}^{\ell}) 
\sum_{s = 0}^{\ell}\overline{c^{\ai,r}(f)} \frac{N(\p)^{s}}{N(a_s)} \frac{N(\eta_{\p\mathfrak{b}^{2}}^{\ell})^{1/2}}{N(\p)^{\ell/2}}e^{2\pi i S\left(\frac{r\widetilde{b_{s}}}{a_{s}}\right)}\chi_f(k_{s})c^{\ai,\frac{r\eta_{\p\mathfrak{b}^{2}}^{\ell}}{a_{s}^{2}}}(f)\\
\notag	&=  \sum_{s = 0}^{\ell}\frac{N(\p)^{s}}{N(a_s)} \frac{N(\eta_{\p\mathfrak{b}^{2}}^{\ell})^{1/2}}{N(\p)^{\ell/2}}e^{2\pi i S\left(\frac{r\widetilde{b_{s}}}{a_{s}}\right)}\chi_f(k_{s})	\sum_{f\in \mathcal B_{\chi,q} \atop \lambda(f)\in \Omega_{t}} \overline{c^{\ai,r}(f)}c^{\ai,\frac{r\eta_{\p\mathfrak{b}^{2}}^{\ell}}{a_{s}^{2}}}(f)\\
%\label{eq:prelim}
\notag	&= \sum_{s = 0}^{\ell}\frac{N(\p)^{s}}{N(a_s)} \frac{N(\eta_{\p\mathfrak{b}^{2}}^{\ell})^{1/2}}{N(\p)^{\ell/2}}e^{2\pi i S\left(\frac{r\widetilde{b_{s}}}{a_{s}}\right)}\chi_f(k_{s})
	\;\delta \Big(r, \tfrac{r\eta_{\p \mathfrak{b}^{2}}^{\ell}}{a_{s}^{2}} \Big)\frac{ 2^d\sqrt {D_F}} %N(\p)^{\ell /2}}
	{\pi^d h_F}\Pl(\Omega_{t})\\\notag & + o(V_{1}(\Omega_{t})).
	\end{align}
%\tfrac{2 v(\ai)\Vol(\Gamma_{0}(\mathfrak{I})\ba G)}{(2\pi)^{d}}	
%	Furthermore, 
By definition,	
	%letting $r' = \frac{r\eta_{\p^{\ell}, \mathfrak{b}}}{a_{s}^{2}}$ in the theorem,  %Theorem \ref{thm: B1 del asian} one has that $\delta(r, r') = 1$ if and only if  $r.r'^{-1} \in (\mathcal{O}_{F}^{\times})^{2}$. we have that  
	$\tilde \delta_{r, {r\eta_{\p \mathfrak{b}^{2}}^{\ell}}/{a_{s}^{2}}} = 1$ if and only if  $  {\eta_{\p \mathfrak{b}^{2}}^{\ell}}/{a_{s}^{2}} \in (\mathcal{O}^{\times}_{F})^+$. %, that is,  if $\frac{a_{s}^{2}}{\eta_{\p^{\ell}, \mathfrak{b}}} = u \in (\mathcal{O}^{\times}_{F}) ^{2}$. 
	We will  see next that this can happen if and only if $2s = \ell$.
	%Veremos que esta condici\'on ocurre s\'olo en el caso en que $2s = \ell$. 
	
	 First, we assume that $\delta_{r, {r\eta_{\p \mathfrak{b}^{2}}^{\ell}}/{a_{s}^{2}}} = 1$. 
	%that is, if
	 This implies  ${a_{s}^{2}}/{\eta_{\p \mathfrak{b}^{2}}^{\ell}} = u \in ({\mathcal{O}^{\times}_{F}})^+$ 
	%then 
	 %$a_{s}^{2}$ y  $\eta_{\p, \mathfrak{b}}^{\ell}$ generan el mismo ideal, i.e. 
 or equivalently	$a_{s}^{2} \mathcal O_F =  \eta_{\p \mathfrak{b}^{2}}^{\ell} \mathcal O_F = \p^{\ell}\mathfrak{b}^{2\ell}$. 
%This  implies that
Thus $\ell$ must be even, and 
	 $v_{\p}(a_{s}) = \tfrac \ell 2 + \ell v_{\p}(\mathfrak{b})$ and $v_{\mathfrak{q}}(a_{s}) = \ell v_{\mathfrak{q}}(\mathfrak{b})$.
	 Since the element $a_s \in F$ % in equation \eqref{eq:as} 
is such that $v_{\p}(a_{s}) = s + \ell v_{\p}(\mathfrak{b})$ 
(see Proposition~\ref{relacion de los tpl})  
%{\red(see \eqref{a_s, b_s})}
 then necessarily $s = \tfrac \ell 2$.

	Conversely, if  $\ell$ is even, let $s = \tfrac{\ell}{2}$. 
%In this case, 
We have that $\p^{\ell/2}\mathfrak{b}^{\ell} =  \eta_{\p \mathfrak{b}^2}^{\ell/2} \mathcal O_F$, that is, we are in the case when $\p^s\mathfrak b^{\ell}$ is principal, and therefore, we can use  \eqref{eq: matrices diagonales} and \eqref{eq: diagonal para s principal}  %in equation \eqref{eq:as} 
 to conclude that  $a_{s} = \eta_{\p \mathfrak{b}^{2}}^{\ell/2}u$ with $u \in \mathcal{O}_{F}^{\times}$. Thus $a_s^2 = \eta_{\p\mathfrak{b}^2}^{\ell}u^2$, that is, ${a_s^2}/{\eta_{\p\mathfrak{b}^2}^{\ell}}=u^2 \in \mathcal{O}_{F}^{\times}$  is totally positive, hence $\delta_{r, {r\eta_{\p\mathfrak{b}^{2}}^{\ell} }/{a_{s}^{2}}}= 1$. 
	
 Furthermore,	in the case when $\ell$ is even, $s = \tfrac{\ell}{2}$, 
%we have $\delta \left(r, \frac{r\eta_{\p^{\ell}, \mathfrak{b}}}{a_{s}^{2}}\right) = 1$. 
by the expression of
%we have $b_{s} = 0$ (see $
$\gamma_{\beta,s}$ in equation \eqref{eq: diagonal para s principal}, we have that $b_s =0$ which implies that $e^{2\pi i S\left( \frac{rb_{s}}{a_{s}}\right)} = 1$. Furthermore, the element $k_{s}$ 
%in  equation \eqref{eq: compacto para s principal} 
satisfies $\chi_f(k_{s}) = 1$.  

Putting all these facts together we get that %\eqref{eq:prelim} equals
\begin{equation*}
	\sum_{f\in \mathcal B_{\chi,q} \atop \lambda(f)\in \Omega_{t}}\!\!\! |c^{\ai,r}(f)|^{2}X_{\ell}(\lambda_{\p}(f)) = %N(\p)^{\ell/2}
% N(\eta_{\p\mathfrak{b}^{2}}^{\ell})^{1/2}e^{2\pi i S\left(\frac{r\widetilde{b_{s}}}{a_{s}}\right)}\chi_f(k_{s})
%\;\delta \Big(r, \tfrac{r\eta_{\p \mathfrak{b}^{2}}^{\ell}}{a_{s}^{2}} \Big)
\frac{ 2^d\sqrt {D_F}}  
{\pi^d h_F}\Pl(\Omega_{t}) + o(V_{1}(\Omega_{t})).	
\end{equation*}
	%Since $a_{s} \in \mathfrak{b} \p^{s}$ and $\eta_{\p^{\ell} ,\mathfrak{b}} \in \p ^{\ell} \mathfrak{b}$, this condition occurs  only in the case when $2 s = \ell$.
	% Now, as the Chebyshev polynomials are orthonormal with respect to the Sato--Tate measure,  we have
Now  $ 0\leq \ell/2 \leq \textrm{ord}_{\p}(r\ai\mathfrak{d})$,  since $r\in \p^\ell \ai^{-1}\mathfrak d^{-1}$. Thus, since  $\ell$ is even  $\Phi_{\ai,r} (X_\ell) = 1$ by  expression \eqref{eq: phi}. This proves the first assertion in the theorem,

Relative to the second assertion, if $\ell$ is odd, then $\Phi_{\ai,r} (X_\ell) = 0$, again by \eqref{eq: phi}. This completes the proof of the theorem.

\end{proof}

As the  Chebyshev polynomials $\lbrace X_{\ell} \rbrace $ are a basis of the space of all polynomials and these are uniformly  dense in $C([-2,2])$, we may  replace $X_{\ell}$ by any    continuous function $f$ in Theorem \ref{thm:principal 1} (see    \cite[Prop.4.8]{BM13} or \cite[Thm. 10.2]{KL13}).
Now, arguing as in \cite[\S 4.3.2]{BM13} to extend  the formula to  characteristic  functions we obtain:

\begin{theorem}\label{thm: principal}
	
	Let $t \mapsto \Omega_{t}$ be a family of sets in $\R^{d}$ as in \eqref{eq: Omega general} %{\blue  with $\Pl(\Omega_{t}) \rightarrow \infty$ as $t \rightarrow \infty$ 
	and let $\mathcal B_{\chi,q}$ be as in Definition \ref{def: BFSchiq}. Let  $\p$ be a prime ideal that is a square in $\mathcal C^+_F$, $\p \nmid \mathfrak{I}$, let $\lambda_{\p}(f)$ be the eigenvalue of $T_{\p}$ on $f\in \mathcal B_{\chi,q}$ 
	and let $\Phi_{\ai,r}(x)$ be as in \eqref{eq: phi}, with $r \in \p^\ell\ai^{-1}\mathfrak d^{-1}$. 
Then, if $t \rightarrow \infty$, given any  interval $I_{\p} \subseteq [-2,2]$, we have that 
	\begin{align}
	\sum_{f\in \mathcal B_{\chi,q} :{\lambda(f)\in \Omega_{t}}, \atop \lambda_{\p}(f) \in I_{\p}} |c ^{\ai,r}(f)|^{2} =  \frac{ 2^d\sqrt {D_F}}{\pi^d h_F}\Pl(\Omega_{t})\Phi_{\ai,r}(I_{\p}) + o(V_{1}(\Omega_{t})).
	\end{align}
\end{theorem}
%\frac{2\sqrt{|D_{F}|}\Vol(\Gamma \ba G)}{(2\pi)^{d}}

%In particular this implies that there are  infinitely many automorphic forms  with the  eigenvalues of $T_{\mathfrak{p}}$ lying in $I_{\mathfrak{p}}$   and with   Casimir eigenvalues in the region $\Omega_{t}$, with densities given by  a polynomial multiple of the  Sato-Tate  measure and by the   Plancherel measure respectively.

 In particular this implies that, if $\p$ is a square in the narrow class group, there are  infinitely many automorphic eigenforms  with  eigenvalues of $T_{\mathfrak{p}}$ lying in $I_{\mathfrak{p}}$, with density given by  $\Phi_{\ai,r}$ that coincides for  $r$ a.e.~ with the Sato-Tate  measure, and with Casimir eigenvalues in the region $\Omega_{t}$, with density given by   the Plancherel measure.

\begin{remark}  This result refines Theorem 1.1 in \cite{BM13} which involves  the eigenvalues $\lambda_{\p^2}(f)$ in place of  $\lambda_{\p}(f)$ and now using a polynomial multiple of  the Sato--Tate measure in place of the  variant used in \cite{BM13}. 
\end{remark}

\section{Applications}
In this section we will give applications of Theorems~\ref{thm:principal 1} and  \ref{thm: principal}, by using some results in \cite{BM10}. % the results in the previous section.
\subsection{Distribution of holomorphic automorphic forms}
%	eigenvalues in  specific regions}
\label{sec: regiones especificas}
%In \cite{BM10} the asymptotic results  on $N^{\kappa, r; \kappa ', r'} (\Omega_t)$ and $\widetilde{N}^{\kappa, r; \kappa ', r'} (\Omega_t)$ are generalized to a  variety of increasing regions  $\Omega_t$. 

The classical case of  holomorphic  forms is of special interest. That is, we  restrict ourselves to   representations $\varpi = \otimes_{j = 1}^{d} \varpi_j$, with  $\varpi_j \in \widehat{G}$  in the discrete series. In this case,  the Casimir eigenvalues are of the form $\lambda_{j}(\varpi) = \tfrac{b_j}{2}\left(1 - \tfrac{b_j}{2}\right) $, with $b_j \in \N$. As often, it is convenient to use, instead of $\lambda(\varpi) \in \R^d$, the corresponding spectral parameter $\nu(\varpi) \in ([0,\infty) \cup i(0, \infty))^d$, $\nu(\varpi) = \sqrt {\tfrac 14 -\lambda(\varpi)}$. We will use a tilde to indicate that the relevant measures like $\widetilde N ^r$, $\widetilde \Pl$ and $\widetilde V_1$ are taken with respect to the variable $\nu$, and we will write
%Let us denote the  function 
\begin{equation}
\widetilde{N}^r (\widetilde{\Omega}) = \sum_{f\in \mathcal B_{\chi,q}\atop \nu_{f} \in \widetilde{\Omega}} |c^{\ai,r} (f)|^{2},
\end{equation}
for sets $\widetilde{\Omega} \subset Y_{\xi}= \prod_{j=1}^d Y_{\xi_j}$, where $\xi_j \in \{0,1\}$,
%\footnote{Check this.} 
\begin{align}\notag
Y_0 = \lbrace \tfrac{b-1}2 : b\ge 2 \textrm { even}\rbrace \;\textrm{ and }\;   %\cup i[0, \infty) \cup (0,\nu_0]\\
 Y_1 = \lbrace \tfrac{b-1}2 : b\ge 2 \textrm { odd}\rbrace. %\cup i[0, \infty). 
\end{align}

Now in light of \eqref{tildePl} and \eqref{measuretildeV1}, if we let $\bold b= \left( \tfrac{b_1 - 1}{2}, \ldots, \tfrac{b_d - 1}{2}\right)$, and 
$\widetilde{\Omega_{\bold b}} =  \lbrace  \bold b \rbrace$, we get in the present case that  $ \widetilde{\textrm{Pl}} (\{\bold b \})   
= \prod_{j = 1}^{d} \left( \tfrac{b_j - 1}{2} \right)$ and $\widetilde{V}_{1}(\{\bold b \}) = \prod_{j = 1}^{d} \left(\tfrac{b_j - 1}{2} \right)^{- A}$, with $A>2$.
%formula \eqref{eq:singleton} in 
As a consequence of Theorem~\ref{thm: principal} we thus obtain

\begin{theorem} \label{thm: caso holomorfo}
Let $\widetilde{\Omega_{\bold b}} =  \{\bold b \}$. Let $\p \nmid \mathfrak{I}$, $\p$  a square in $\mathcal C_F^+$, %the narrow  class group,  
%$\lambda_{\p}(f)$ the eigenvalue of $T_{\p}$ for   $\varpi$, 
let $I_{\p} \subseteq [-2,2]$ be an interval, and let $\Phi_{\ai,r}$ be as in \eqref{eq: phi}, with $r\in \p^\ell \ai^{-1}\mathfrak d^{-1}$. 
 Then, if the product $\prod_{j = 1}^{d} \tfrac{b_j - 1}{2}$ tends to infinity we have 
	\begin{align}
\widetilde{N}(\{\bold b \}, I_{\p}) =  \sum_{f\in\mathcal B_{\chi,q} : \nu_{f}\in \{\bold b \} \atop \lambda_{\p}(f) \in I_{\p}} \!\!\!|c^{\ai,r} (f)|^{2} =& \frac{2^d\sqrt{|D_{F}|}}{(2\pi)^{d}}\prod_{j = 1}^{d} \left( \tfrac{b_j - 1}{2} \right)\Phi_{\ai,r}(I_{\p})+ o{\left(\prod_{j = 1}^{d} \left(\tfrac{b_j - 1}{2} \right)\right)\!\!}^{- A},
	\end{align}
with $A>2$.
\end{theorem}

%\begin{remark}
%En Capitulo \ref{chp: comparacion con resultados previos} compararemos este resultado con la Proposition \ref{thm: KL 08 asintotico}.
%\end{remark}
%\section{Equidistribucion pesada de autovalores de Hecke}\label{sec: equidistribucion con peso}

%En este capitulo daremos un resultado de equidistribucion pesada que complementa los obtenidos en \cite{Se97} y \cite{KL08}. 

\subsection{Weighted equidistribution of Hecke eigenvalues}

As an application of Theorem \ref{thm:principal 1} we also obtain a  result of weighted equidistribution of Hecke eigenvalues. 
\begin{definition}	Let $(X, \mu)$ be a Borel measure space. Let $S_{1}, S_{2}, \ldots$ be a sequence of non-empty finite subsets  of $X$.
	%, and let $|S_{i}|$ be the cardinality of $S_{i}$.	
	Then  one says that the sequence  $S_{i}, \, i \in \N$ is \emph{equidistributed}  with respect to  $d\mu$ (or $\mu$-equi\-distri\-buted) if for any  continuous function  $g$ on  $X$ we have
	\begin{equation}
	\lim_{ i\rightarrow \infty} \frac{\sum_{x\in S_{i}} g(x)}{|S_{i}|} = \int_{X} g(x) d\mu (x).
	\end{equation}
\end{definition}

\begin{example} If $S_{i}=\lbrace 0, 1/i, 2/i, \ldots, 1 \rbrace$ then $\lbrace S_{i} \rbrace_{i\in\N}$ is equi\-distri\-buted with respect to the   Lebesgue measure on $X=[0,1]$.
\end{example}	
If $a_{1}, a_{2}, \ldots$ is a sequence of points on $X$,  then $\lbrace a_{i} \rbrace$ is  $\mu$-equi\-distri\-buted %or \emph{equidistributed  with weight w} 
if $\lbrace S_{i} \rbrace$ is $\mu$-equi\-distri\-buted, where $S_{i} = \lbrace a_{1}, \ldots, a_{i} \rbrace$. 
	
	If to each  element $x \in S_{i}$ we asign a weight $w_{x} \in \mathbb{R}^{+}$, then the sequence $\lbrace S_{i} \rbrace$ is $w$-\emph{equidistributed} if for any  continuous  function $g: X \rightarrow \mathbb{C}$, we have:
	\begin{equation}
	\lim _{i \rightarrow \infty} \frac{\sum_{x \in S_{i}} w_{x}g(x)}{\sum_{x \in S_{i}} w_{x}} = \int_{X} g(x)d\mu (x).
	\end{equation}

%\begin{theorem}
%If $\p \nmid \mathfrak{I}$, we have that 
%\begin{equation}
%\lim_{t \rightarrow \infty} \frac{\sum_{\varpi, \lambda_{\varpi}\in \Omega_{t}} |c ^{r}(\varpi)|^{2}f(\lambda_{\p, \varpi})}{\sum_{\varpi, \lambda_{\varpi}\in \Omega_{t}} |c ^{r}(\varpi)|^{2}}= \int_{\R}f(x)\Phi_r(x).
%\end{equation}
%\end{theorem}
%\begin{proof}
%We take  $l=0$ in (\ref{principal}): 
%\begin{align*}
%\sum_{\varpi, \lambda_{\varpi}\in \Omega_{t}} |c ^{r}(\varpi)|^{2} = \frac{|N(a_{\mathfrak{q}})|}{\sqrt{N|\eta_{\p,{\mathfrak{b}}}|}}e^{2\pi i S\left(\frac{rb_{\mathfrak{q}}}{a_{\mathfrak{q}}}\right)}\frac{2\sqrt{|D_{F}|}\Vol(\Gamma \ba G)}{(2\pi)^{d}}Pl(\Omega_{t}) + o(V_{1}(\Omega_{t}))
%\end{align*}
%Then, for any  $l \geq 0$ we have
%\begin{align*}
%\lim_{t \rightarrow \infty} \frac{\sum_{\varpi, \lambda_{\varpi}\in \Omega_{t}} |c ^{r}(\varpi)|^{2}X_{l}(\lambda_{\p, \varpi})}{\sum_{\varpi, \lambda_{\varpi}\in \Omega_{t}} |c ^{r}(\varpi)|^{2}}= 1 = \int_{\R}X_{l}(x)\Phi_r(x)
%\end{align*}
%if $l=2l'$.

%Since $\lbrace X_{l} \rbrace$ generates the space of all  polynomials, and this space  is dense in   $L^{\infty}([-2,2])$, we may replace   $X_{l}$ by any  continuous function, hence the theorem follows.
%\end{proof}

\begin{theorem}\label{thm: equidistribucion con peso}
	Let $t \mapsto \Omega_{t}$ be a family of subsets in $\R^{d}$ as in \eqref{eq: Omega general}. Let $\p \nmid \mathfrak{I}$, $\p$ a square in $\mathcal C_F^+$ %the narrow class group,
	 %$\lambda_{\p} (f)$ the eigenvalue of $T_{\p}$ for  $ f$ 
	 and let $\Phi_{\ai,r}$ be as in \eqref{eq: phi} with $r\in \p^\ell \ai^{-1}\mathfrak d^{-1}$. %  = \sum_{\ell'=0}^{\textrm{ord}_{\p}(r\mathfrak{d})} X_{2\ell'}(x)d\mu_{\infty}(x)$. 
	  Then, if  $g$ is any continuous function on $\R$, as   $t \rightarrow \infty$ we have
	 
	\begin{equation}
	\lim_{t \rightarrow \infty} \frac{\sum_{_{f\in \mathcal B_{\chi,q} \atop \lambda(f)\in \Omega_{t}}} |c^{\ai,r} (f)|^{2}g(\lambda_{\p}(f))}{\sum_{_{f\in \mathcal B_{\chi,q} \atop \lambda(f)\in \Omega_{t}}} |c^{\ai,r} (f)|^{2}}= \int_{\R}g(x)\Phi_{\ai,r}(x).
	\end{equation}
\end{theorem}
\begin{proof}
If we let  $\ell=0$ in Theorem~\ref{thm:principal 1} then we get
%\begin{align}
%\sum_{\varpi, \lambda_{\varpi}\in \Omega_{t}} |c ^{r}(\varpi)|^{2} = \frac{|N(a_{\mathfrak{q}})|}{\sqrt{N|\eta_{\p,{\mathfrak{b}}}|}}e^{2\pi i S\left(\frac{rb_{\mathfrak{q}}}{a_{\mathfrak{q}}}\right)}
%\frac{2\sqrt{|D_{F}|}\Vol(\Gamma \ba G)}{(2\pi)^{d}}Pl(\Omega_{t}) + o(V_{1}(\Omega_{t})).
%\end{align}
\begin{equation}\label{eq:ell=0}
\lim_{t \rightarrow \infty} \sum_{_{f\in \mathcal B_{\chi,q} \atop \lambda(f)\in \Omega_{t}}} |c^{\ai,r} (f)|^{2}%g(\lambda_{\p}(f))
=\frac{2^d\sqrt{|D_{F}|}}{\pi^{d}h_F}Pl(\Omega_{t}) + o(V_{1}(\Omega_{t})).
%\sum_{_{f\in \mathcal B_{\chi,q} %\atop \lambda(f)\in \Omega_{t}}}|c^{\ai,r} (f)|^{2}}
%= \int_{\R}g(x)\Phi_{\ai,r}(x).
\end{equation}

For $\ell > 0$ we have
\begin{equation}\label{eq:ell>0}
\sum_{f\in \mathcal B_{\chi,q} \atop \lambda(f)\in \Omega_{t}} \!\!\!\!\!\!|c^{\ai,r} (f)|^{2}X_{\ell}(\lambda_{\p}(f))\!\!= \begin{array}{ll}
\!\!\!\! \Bigg\lbrace \! 
\begin{array}{ll}
\frac{2^d\sqrt{|D_{F}|}}{\pi^{d}h_F}\Pl(\Omega_{t})\Phi_{\ai,r}(X_{\ell}) + o(V_{1}(\Omega_{t}))\;\; \textrm{ if $\ell$ even} % = 2\ell '$ for $0\leq \ell '\leq \textrm{ord}_{\p}(r\mathfrak{ad})$} 
\\ o(V_{1}(\Omega_{t})) \; \; \textrm{ if  $\ell$ odd}.
\end{array}
\end{array}
\end{equation}

By taking the quotient of \eqref{eq:ell>0} by \eqref{eq:ell=0} we obtain  

\begin{align*}
\lim_{t \rightarrow \infty} \frac{\sum_{_{f\in \mathcal B_{\chi,q} \atop \lambda(f)\in \Omega_{t}}} |c ^{\ai, r}({ f})|^{2}X_{\ell}(\lambda_{\p}(f))}{\sum_{_{f\in \mathcal B_{\chi,q} \atop \lambda(f)\in \Omega_{t}}} |c ^{\ai, r}({f})|^{2}}= \int_{\R}X_{\ell}(x)\Phi_{\ai,r}(x) = \begin{array}{ll}
\!\!\!\! \Bigg\lbrace \! 
\begin{array}{ll}
1 \; \; \textrm{ if  $\ell$ even}  \\ 0 \; \; \textrm{ if $\ell$ odd}.
\end{array}
\end{array}
\end{align*}

As $\lbrace X_{\ell} \rbrace$ generates the space of all  polynomials and this space is uniformly dense in $C([-2,2])$, we may replace   $X_{\ell}$ by any  continuous function $g$ and the theorem follows.
\end{proof}

\begin{remark}\label{coro: densos} The above result implies that if $t_n$ is an increasing  sequence tending to $+\infty$, then  the set $\bigcup_{n \in \N}\lbrace\lambda_{\p}(f): f\in \mathcal B_{\chi,q},  \lambda(f)\in \Omega_{t_n} \rbrace$ is a dense subset of $[-2,2]$. 
\end{remark}
\begin{remark}
	
	Theorem~\ref{thm: caso holomorfo} 
	% Theorem~\ref{thm: equidistribucion con peso} 
	 extends results by  Serre  \cite{Se97} for holomorphic forms for $F=\Q$ and  by  Knightly--Li \cite{KL08} for  holomorphic forms for $F$ a totally real number field. We note that  Theorem~\ref{thm: equidistribucion con peso}  includes all  representations (not only the discrete series) and involves a spectral  parameter tending to infinity in place of the level, as is the case in  \cite{Se97} and \cite{KL08}.
\end{remark}

\section {Appendix}

In this section we will explain the estimates necessary in the proof of Theorem~\ref{thm: asymptotic formula}. We will  only give a sketch of the proofs, the arguments are similar to those given in \cite{BM10} (using some facts proved in \cite{BMP03}).

\subsection {Kloosterman term}
We let as before $\{1,\ldots,d\}= E\sqcup Q_+ \sqcup Q_-$ where $Q_+ \sqcup Q_- \neq \emptyset$  and use test functions $\varphi = \otimes_{j=1}^d \varphi_j$, with $\varphi_j$ as in \cite[\S 2.1.1 and Lemma 2.2]{BM10} and  $\varphi_E = \otimes_{j \in E} \varphi_j$.   

We need to bound the sum 

	\begin{align} \label{eq:Kloostterm}
  \frac{ 2^{d-1}}{h_F} \sum_{\mathfrak{c} : \mathfrak{c}^2 \sim \mathfrak{a}^2} \sum_{\varepsilon \in \mathcal{O}^* / (\mathcal{O}^*)^2} \sum_{c \in \mathfrak{c}^{-1}\mathfrak{I}} B\varphi(\nu(f), \tfrac{\varepsilon rr'}{c^2 \gamma}) \frac{KS (r, \mathfrak{a} ; \varepsilon r' \gamma^{-1}, \ai;c,\mathfrak c)}{N(c\mathfrak{c})}
\end{align}% 
where $\gamma \in \mathcal O_F$ is such that $\mathfrak c^2 = \langle\gamma\rangle \ai^2$.
It will be sufficient to bound the inner sum since the other two sums are finite. We will use methods and results from \cite[\S2.2] {BM10} or \cite[\S4 and \S5]{BMP03}. 

Set  $\tau\in (1/4,1/2)$, $\rho \in (1-\tau,1)$, $\gamma\in (\tau, 1/2)$, $\rho_1= 3/2 - \gamma -\tau \in (1/2,1)$, $U\ge 1$ and $A_1>0$.
Then, by  \cite[(52)]{BM10} or \cite[Lemma 4.3]{BMP03}  we have that 
\begin{equation}\label{eq:Besselbound}| B\varphi(\nu(f), \tfrac{\varepsilon rr'}{c^2 \gamma})| \ll_{A_1} 
\prod_j \min\left(a_j\left (\frac{4\pi |r_jr'_j|^{1/2}}{|c_j||\gamma|^{1/2}}\right)^{2\tau}, b_j\right)
\end{equation}
where
\begin{table}[h]
\centering
\begin{tabular}{|c|c|c|}
\hline
$a_j$ & $b_j$ &   \\
\hline
$N_j(\varphi_j)$  & $N_j(\varphi_j)$  & $j \in E$ \\
\hline 
$|q|^{-A_1}$ & $|q|$ & $j \in Q_-$ \\
\hline
$e^{\tfrac 1 2 \tau^2 U}|q|^{\rho_1}$ & $|q|$ & $j\in Q_+$  \\ 
\hline
\end{tabular}
\

\end{table}

\ 

and $\|\varphi_E\|_E= \prod_{j \in E}
 \|\varphi_j\|$, $\|\varphi_j\|= \sup_{\nu, 0 \leq \textrm{ Re } \nu \leq \tau} |\varphi(\nu)|(1+ | \nu |)^a + \sum_{ b\equiv \xi_j (2), b \geq 2}  b^a | \varphi_j(\tfrac{b-2} 2)|$.

For the second factor we use the Weil bound:
\begin{equation}\label{eq:Kloostbound}
\frac {KS (r, \mathfrak{a} ; \varepsilon r' \gamma^{-1}, \ai;c,\mathfrak c)}{N(c\mathfrak c)} \leq \frac{|N(gcd(r\ai\mathfrak d, r' \gamma^{-1}\mathfrak c^2 \ai^{-1}\mathfrak d, c\mathfrak c))|^{1/2}|N(c\mathfrak c)|^{1/2 + \varepsilon}}{|N(c\mathfrak c)|}\ll |N(c\mathfrak c)|^{\varepsilon -1/2}
\end{equation}

Now we substitute  \eqref{eq:Besselbound} and \eqref{eq:Kloostbound} in the inner sum in \eqref{eq:Kloostterm} and obtain that this sum is bounded by 
\begin{align}
\notag &\sum_{c \in \mathfrak{c}^{-1}\mathfrak{I}} \prod_j \min\left(a_j\left (\frac{4\pi |r_jr'_j|^{1/2}}{|c_j||\gamma|^{1/2}}\right)^{2\tau}, b_j\right)|N(c\mathfrak c)|^{\varepsilon -1/2}\ll \\
 & N(\mathfrak c)^{\varepsilon -1/2}\sum_{\langle c\rangle\subseteq \mathfrak{c}^{-1}\mathfrak{I}}\prod_{\p\nmid \mathfrak I} N(\p)^{v_{\p} (c)(\varepsilon -1/2)} \prod_{\p\mid \mathfrak I} N(\p)^{v_{\p} (c)\varepsilon } \sum_{\zeta \in\mathcal O^\times}\prod _j \min (p_j |\zeta_j|^{-2\tau}, q_j)
\end{align}
where $p_j=a_j\left( \frac{4\pi |r_jr'_j|}{|\gamma|^{1/2}} \right)^{2\tau}$, $q_j = b_j$, $c \mathcal O_F = \prod_{\p \textrm{ prime}} \p^{v_{\p} (c)}$.

Now we use  \cite[Lemma 4.3]{BMP03} or  \cite[Lemma 2.2]{BM09} with $\alpha = 2\tau$, $\beta=0$  and $y_j = c_j ^{-1}$ to estimate the sum over $\zeta \in \mathcal O^\times_{F}$ by  
\begin{equation}\ll (1 + |\log|N(c)| + \tfrac 1{2\tau}\log \tfrac{Q}{P}|)^{d-1}
\min(P|N(c)|^{2\tau},Q)
\end{equation}
where $P = \prod p_j$ and $Q =\prod q_j$. Now arguing like in \cite[p.3857]{BM10} (or in \cite[p.703]{BMP03}) we get that   
\begin{align}  
 \sum_{\zeta \in\mathcal O^\times}\prod _j \min (p_j |\zeta_j|^{-2\tau}, q_j)\ll &  e^{\tfrac 12 \tau^2 U|Q_+|(1 + \varepsilon)} \|\varphi_E\|_E|N(c)|^{-2\tau(1-\varepsilon)}.\\
& \notag \prod_{j\in Q_+}|q_j|^{\rho + (1-\rho_1)\varepsilon} 
\prod_{j\in Q_-}|q_j|^{-A_1 + (A_1-1)\varepsilon}. 
\end{align}

In this way, we find the following estimate for the Kloosterman term:
\begin{align*}
\ll & N(\mathfrak{c})^{\varepsilon - 1/2}  e^{\tfrac 12 \tau^2 U|Q_+|(1 + \varepsilon)} \|\varphi_E\|_E \prod_{j\in Q_+}|q_j|^{\rho + (1-\rho_1)\varepsilon} 
\prod_{j\in Q_-}|q_j|^{-A_1 + (A_1-1)\varepsilon}.\\
& \sum_{ c \mathcal O_F\subset \mathfrak{c}^{-1}\mathfrak{I}} \prod_{\p\nmid \mathfrak I} N(\p)^{v_{\p} (c)(\varepsilon -1/2 - 2\tau(1-\varepsilon))} \prod_{\p\mid \mathfrak I} N(\p)^{v_{\p} (c)(\varepsilon- 2\tau(1-\varepsilon)) } \\
& \ll N(\mathfrak{c})^{\varepsilon - 1/2}  e^{\tfrac 12 \tau^2 U|Q_+|(1 + \varepsilon)} \|\varphi_E\|_E \prod_{j\in Q_+}|q_j|^{\rho + (1-\rho_1)\varepsilon} 
\prod_{j\in Q_-}|q_j|^{-A_1 + (A_1-1)\varepsilon}.\\
&\prod_{\p \nmid \mathfrak{I}} \frac{1}{1-N(\p)^{\varepsilon -1/2 - 2\tau(1-\varepsilon)}} . \prod_{\p \mid \mathfrak{I}} \frac{1}{1 - N(\p)^{\varepsilon- 2\tau(1-\varepsilon)}}.
\end{align*}
Under the additional assumption on $\varepsilon$ that $2\tau(1-\varepsilon) + \tfrac 1 2 - \varepsilon > 1$, that is, $\varepsilon(\tau-1/4)< \tau +1/2$, the product converges and  the Kloosterman term is bounded by
\begin{align} \label{eq:bound Kloosterman}
\notag&\ll_{F,\mathfrak{I}, r, r', \varepsilon}   e^{ \tfrac{ \tau^2 (1+\varepsilon) U|Q_+|}2} \|\varphi_E\|_E \prod_{j\in Q_+}|q_j|^{\rho_1(1-\varepsilon)+ \varepsilon} 
\prod_{j\in Q_-}|q_j|^{-(A_1(1-\varepsilon)+ \varepsilon)}\\
&\ll_{F,\mathfrak{I}, r, r', \varepsilon}
 e^{ t_0 U|Q_+|} \|\varphi_E\|_E \prod_{j\in Q_+}|q_j|^{\rho} 
\prod_{j\in Q_-}|q_j|^{-A}
\end{align}
where  $\rho = \rho_1+(1-\rho_1)\varepsilon$,  $A=A_1 + (1-A_1)\varepsilon$ and $t_0 =  \tfrac{ \tau^2 (1+\varepsilon)}2$.
 
\

\subsection{Eisenstein contribution}
We now show that the Eisenstein contribution to the asymptotic formula is of smaller order than that of  the Kloosterman term.
In the first place, if $Q_{-} \ne \emptyset$ the Eisenstein term is just zero (\cite[\S5.3]{BMP03}) so we may assume that $Q = Q_+$.

The Eisenstein series has the form
\begin{equation}
E(\chi,\nu, i\mu, g,q)  = \sum_{\gamma \in {\GL_2(F)}_P\backslash \GL_2(F)} \psi(\gamma g)^{\nu+i\mu+\rho} \phi_q(k_\infty (\gamma g)) \chi(k_f(\gamma g) )
\end{equation}
where $\psi = \otimes_v \psi_v$, with  $ \psi_v \left ( \left(\begin{smallmatrix} a&b\\c&d \end{smallmatrix}\right) k\right) =|a/d|_v$  and $k \in K_v$, for  $v$ a finite or infinite valuation of $F$. By the product formula $\psi$ is left invariant by   $ \GL_2(F)_P$.
The classical components $E_\ai$ of $E$ are defined as in \eqref{classical components}.

Thus, the Eisenstein contribution to  the sum  formula \eqref{eq:sum formula} is a sum of terms of the form
\begin{equation}\label{eq:CSC}
\sum_{\mu \in \mathcal L_\chi} \int_{-\infty}^\infty \varphi(it + i\mu)  D^{\ai,r}(\chi,it, i\mu) \overline{ D^{\ai,r'}(\chi,it,i\mu)}\, dt
\end{equation}
where $\varphi$ is a test function as in \cite[\S 2.1.1 and Lemma 2.2]{BM10}, $ D^{\ai,r}(\chi,\nu,i\mu)$ is the Fourier coefficient of order $r$ of $E_\ai (\chi,\nu, i\mu, g_\infty,q)$ and $\mathcal L_\chi$ is a lattice in the hyperplane $\sum_{j=1}^d x_j = 0$ in $\R^d$.

We argue as in \cite[\S5]{BMP03}. Since  the group $\Gamma_0(\mathfrak I,\mathfrak a)$ contains the principal congruence subgroup $\Gamma(\mathfrak I \mathfrak a)$, then the Eisenstein series with respect to $\Gamma_0(\mathfrak I,\mathfrak a)$ is a linear combination of Eisenstein series for $\Gamma(\mathfrak I \mathfrak a)$. Now to estimate the Fourier coefficients for $\Gamma(\mathfrak I \mathfrak a)$ we argue as in \S 5.1 and \S5.2 of \cite{BMP03}. 
As a first step, we express the Fourier coefficients in terms of $L$-series $L(s,\lambda_\mu,\chi)$. Secondly, we prove a bound for $\textrm {Re}\, s =1$ of the type 
\begin{align*}
1/ L(t,\lambda_\mu,\chi) &\ll\log^7(2 + t)+ \log\|\mu\|) \quad {\textrm { if } }\mu\ne 0\\
1/ L(t,\lambda_0,\chi) &\ll\log^7(2 + t), 
\end{align*} 
which implies that 
\begin{align*}
D^r(\chi,it,i\mu) &\ll N(r)^\varepsilon ( \log^7(2 + t)+ \log\|\mu\|) \quad {\textrm { if } }\mu\ne 0 \\
D^r(\chi,it,0) &\ll N(r)^\varepsilon (\log^7(2 + t).
\end{align*} 
 Now we may bound \eqref{eq:CSC} by
\begin{equation}\label{eq:bound Eisenstein}
\sum_{\mu \in \mathcal L_\chi} N(r)^\varepsilon \|\varphi_E\|_E\prod_{j\in E} \int_{-\infty}^\infty  (1+ (t_j+ \mu_j)^2)^{\varepsilon-a/2} \,\prod_{ j \in Q_+} \int_{-\infty}^\infty  (1+ (t_j+ \mu_j)^2)^{\varepsilon} e^{-U(t + \mu-|q|^2)} dt\end{equation}
with the assumption that $a/2-\varepsilon>1 $. Now arguing as in \cite[p.3859]{BM10} we obtain that the Eisenstein term is 
\begin{equation}
\ll_{U,\varepsilon} \|\varphi_E\|_E\prod_{ j \in Q_+} |q_j|^{2\varepsilon},
\end{equation}
which is absorbed by the Kloosterman term \eqref{eq:bound Kloosterman}.
This gives the estimate for the Eisenstein contribution.

Thus, we have shown that both contributions are estimated by
\begin{equation}
\ll_{r,r',U} \|\varphi_E\|_E\prod_{j\in Q} |q_j|. 
\end{equation}
Now the comparison with the Plancherel measure and the proof of the asymptotic formula is completed by arguing  as in \cite[sections 3-5]{BM10}.

\bibliographystyle{plain}

\end{document}